\documentclass{amsart}
\usepackage{amsmath}
\usepackage{amsfonts}
\usepackage[mathscr]{eucal}
\usepackage{amssymb}
\usepackage{marginnote}
\newtheorem{thm}{Theorem}[section]

\newtheorem{cor}[thm]{Corollary}

\newtheorem{lemma}[thm]{Lemma}
\newtheorem{prop}[thm]{Proposition}

\theoremstyle{definition}
\newtheorem{defn}[thm]{Definition}
\newtheorem{remark}[thm]{Remark}
\newtheorem{exam}[thm]{Example}

\newcommand{\ff}[1]{\mathfrak{#1}}

\newcommand{\inner}[2]{\left\langle {#1}, {#2} \right\rangle}

\usepackage[all]{xy}
\begin{document}

\title[A Spectral Characterization of $\mathcal{AN}$ operators]{A Spectral Characterization of $\mathcal{AN}$ Operators}

\author[S.~K.~Pandey And V.~I.~Paulsen]{Satish K.~Pandey And Vern I.~Paulsen}
\address{Department of Pure Mathematics, University of Waterloo,
Ontario, Canada N2L 3G1}
\email{satish.pandey@uwaterloo.ca}
\address{Department of Pure Mathematics, University of Waterloo,
Ontario, Canada N2L 3G1}
\email{vpaulsen@uwaterloo.ca}
\date{January 5, 2015}

\begin{abstract}
We establish a spectral characterization theorem for the operators on
complex Hilbert spaces of arbitrary dimensions that attain their norm on every closed subspace. The class of these operators is not closed under addition. Nevertheless, we prove that the intersection of these operators with the positive operators forms a proper cone in the real Banach space of hermitian operators.
\end{abstract}

\maketitle


\section{Introduction} 

Throughout this paper $\mathcal{H}$ and $\mathcal{K}$ will denote complex
Hilbert spaces and we write $\mathcal{B}(\mathcal{H},\mathcal{K})$ for the set of all
bounded, linear operators from $\mathcal{H}$ to $\mathcal{K}$. We recall that $\mathcal{B}(\mathcal{H},\mathcal{K})$ is a complex Banach space with respect to the operator norm
$$
\|T\| = \text{sup}\{\|Tx\|: x\in \mathcal{H}, \|x\| \leqslant 1\}.
$$
\begin{defn}
An operator $T\in \mathcal{B}(\mathcal{H},\mathcal{K})$ is said to be an $\mathcal N$
operator or to satisfy the property $\mathcal{N}$ if there is an element $x$ in the unit sphere of $\mathcal{H}$ such that $\|T\|=\|Tx\|$.
\end{defn}
Such operators achieve their norm and hence are known as {\it norming}
 operators. A generalization of the property $\mathcal{N}$ leads to a new class of operators in $\mathcal{B}(\mathcal{H},\mathcal{K})$. 
\begin{defn}
An operator $T\in \mathcal{B}(\mathcal{H},\mathcal{K})$ is said to be an $\mathcal{AN}$ operator or to satisfy the property $\mathcal{AN}$, if for every nontrivial closed subspace $\mathcal{M}$
of $\mathcal{H}$, $T|_{\mathcal{M}}$ satisfies the property $\mathcal{N}$.
\end{defn}
Alternatively, an operator $T\in \mathcal{B}(\mathcal{H},\mathcal{K})$ is said to be
an $\mathcal{AN}$ operator if for every nontrivial closed subspace
$\mathcal{M}$ of $\mathcal{H}$ there is an element $x\in \mathcal{M}, \|x\|=1$ such
that $\|T|_{\mathcal{M}}\|=\|T|_{\mathcal{M}}(x)\|$. Since these operators, when
restricted to any nontrivial closed subspace of $\mathcal{H}$, achieve
their norm on that closed subspace, we say that these operators are {\it
  absolutely norming} and hence the name  $\mathcal{AN}$ operator. Needless
to say,  every $\mathcal{AN}$ operator is an $\mathcal{N}$ operator. \\

The $\mathcal{AN}$ operators were introduced and studied in \cite{CN}
and \cite{R}. Carvajal and Neves \cite{CN} proved a partial structure theorem \cite[Theorem 3.25]{CN} for the class of positive $\mathcal{AN}$ operators on complex Hilbert spaces that included an uncharacterized ``remainder" operator. This theorem motivated Ramesh \cite{R}
to attempt to obtain a full characterization theorem \cite[Theorem 2.3]{R}, without remainder, for positive $\mathcal{AN}$ operators on separable complex Hilbert spaces.  

In this paper, we present a counterexample to Ramesh's characterization theorem \cite[Theorem 2.3]{R}. We then give a full spectral characterization of the class of positive  $\mathcal{AN}$ operators on complex Hilbert spaces of arbitrary dimensions, earlier results needed to assume separability. The correct characterization requires more terms than were used in \cite{R} and \cite{CN}. 
Using this theorem, we prove a full characterization theorem for the class of $\mathcal{AN}$ operators on complex Hilbert spaces of arbitrary dimension.

We begin by giving noninductive proofs of some basic facts, which allows us to remove the separability assumption. In section 3 and 4, we derive necessary and sufficient conditions for a positive operator that satisfies the $\mathcal{AN}$ condition and consequently establish a spectral characterization theorem (see \ref{SpThPAN}) for these operators in section 5. This theorem, together with the polar decomposition theorem, then paves the way for our main result in section 6: the full spectral characterization of the class of $\mathcal{AN}$ operators (see \ref{SpThAN}). The class of these operators is not closed under addition. Nevertheless, we prove that the class of positive $\mathcal{AN}$ operators is a proper cone in the real Banach space of hermitian operators.

We end this section by presenting a counterexample to \cite[Theorem 2.3] {R}. 
\begin{exam}
Consider the operator
$$
T=\begin{bmatrix}
 \frac{1}{2}\\
    & 1 & & & \text{\huge0}\\
    & & 1\\
    & & & \ddots &\\
    & \text{\huge0} & & & 1\\
    & & & & & 1
\end{bmatrix} \in \mathcal{B}( l^2).
$$
That $T$ is positive operator on a separable Hilbert space is obvious. $T$ is not compact. The infimum of the eigenvalues of this operator, Ramesh's $m(T)=1/2$. The operator $T-m(T)I=\text{diag}(0,1/2,1/2,...)$ is not compact. Consequently, $T$ is neither compact nor of the form $K + m(T)I$ for some positive compact operator $K$. Even more, there does not exist $\alpha \geq 0$ such that $T=K+\alpha I$ for some positive compact operator $K$. Thus, if \cite[Theorem 2.3]{R} was correct, then $T$ would not satisfy $\mathcal{AN}.$

However, we now prove that $T$ satisfies the property $\mathcal{AN}$. Suppose that $\mathcal{M}$ is an arbitrary nontrivial closed subspace of $\mathcal{H}$. If $\mathcal{M}$ is one dimensional, then $T|_{\mathcal{M}}$ attains its norm at any vector in $\mathcal{M}$ with unit norm. 
If dim$(\mathcal{M})\geq 2$ and $\mathcal{M}$ contains two noncollinear vectors which are nonzero in the first entry, then there exists a linear combination of these two vectors with $0$ in the first entry. Letting $x_0$ be the normalization of this vector, we get 
$1=\|x_0\|=\|T(x_0)\|\leq \|T|_{\mathcal{M}}\|\leq\|T\|=1$ and so we have equality throughout and $T$ attains its norm on $\mathcal{M}$.\\
Finally, if dim$(\mathcal{M})\geq 2$ and it does not contain any two such vectors, then it either has a single such vector and its scalar multiples or no such vector. Since dim$(\mathcal{M})\geq 2$, $\mathcal{M}$ has at least one vector linearly independent from all vectors with nonzero first entry and that vector must have $0$ in its first entry. If we normalize this vector --- we call this vector $x_0$ --- we get $1=\|x_0\|=\|T(x_0)\|\leq \|T|_{\mathcal{M}}\|\leq\|T\|=1$ and hence $T$ attains its norm on $\mathcal{M}$. This proves the assertion and serves to be a counterexample to the characterization Theorem $2.3$ of  \cite {R}.
\end{exam}

While some of our results have parallels in \cite{CN} and \cite{R}, our proofs are quite different, since we do not assume separability and avoid representing operators by ordered series indexed by $\mathbb N.$

\section{Preliminaries} 

\begin{prop}
If $T\in \mathcal{B}(\mathcal{H},\mathcal{K})$ is a compact operator, then $T$ satisfies the property $\mathcal{AN}$.
\end{prop}

\begin{proof} If $T$ is a compact operator from $\mathcal H$ to $\mathcal K$
  then the restriction of $T$ to any closed subspace $\mathcal M$ is a
  compact operator from $\mathcal M$ to $\mathcal K.$ So it will be sufficient
  to prove that if $T$ is a compact operator then $T$ satisfies $\mathcal
  N.$

Let $B = \{x\in \mathcal{H}: \|x\|\leqslant 1\}$ be the closed unit ball of
$\mathcal{H}$. Since $T$ is a compact operator, $T(B)$ is a compact subset of
$\mathcal{K}$ in the norm topology \cite[page 55]{KR}. Also, $\|\cdot\|_{\mathcal{K}}:T(B)\longrightarrow [0,\infty)$ is a continuous function on $T(B)$. Consequently we have $\sup\{\|Tx\|_{\mathcal{K}}:\|x\|_{\mathcal{H}}\leqslant 1\} = \max\{\|Tx\|_{\mathcal{K}}:\|x\|_{\mathcal{H}}\leqslant 1\}$. It therfore implies that there exists $x_0\in B$ such that
$\|T\| = \|Tx_0\|_{\mathcal{K}}$. This, together with $\|Tx_0\|_{\mathcal{K}} \leqslant \|T\|\|x_0\|_{\mathcal{H}} \leqslant \|T\|$, implies that $\|x_0\|_{\mathcal{H}}=1.$ This proves the proposition.
\end{proof}


\begin{prop}\cite[Proposition 2.3]{CN}\label{norming}
Let $T\in \mathcal{B}(\mathcal{H})$ be a self-adjoint operator. Then $T$ satisfies $\mathcal{N}$ iff $\|T\|$ or $-\|T\|$ is an eigenvalue of $T$.
\end{prop}
%
%

This result leads us to the following theorem.

\begin{thm} \label{!}
Let $T\in \mathcal{B}(\mathcal{H})$ be a positive operator. Then $T$ satisfies $\mathcal{N}$ iff $\|T\|$ is an eigenvalue of $T$.
\end{thm}

\begin{thm}\label{TiffT*T}
Let $\mathcal{H}$ and $\mathcal{K}$ be complex Hilbert spaces and $T\in \mathcal{B}(\mathcal{H},\mathcal{K})$.
Then $T$ satisfies $\mathcal{N}$ iff $T^*T$ satisfies $\mathcal{N}$. 
\end{thm}

\begin{proof} First assume that $T$ satisfies $\mathcal{N}$. There exists $x$ in the unit sphere of $\mathcal{H}$ such that $\|Tx\|=\|T\|.$ Then 
\[\|T^*T\| = \|T\|^2 = \langle Tx, Tx \rangle = \langle T^*Tx,x \rangle \le \|T^*Tx\| \le \|T^*T\|,\]
and so we have equality throughout which implies that $T^*T$ satisfies $\mathcal{N}$.

Conversely, if $T^*T$ satisfies $\mathcal{N}$, then by Theorem \ref{!} $\|T^*T\|$ is an eigenvalue of $T^*T$. Suppose $y\in \mathcal{H}$ is the corresponding eigenvector of unit norm. Then
$ \|Ty\|^2 = \langle T^*Ty,y \rangle =\langle \|T^*T\|y,y\rangle = \|T\|^2,$
and the result follows.
\end{proof}
 
Let $T \in B(\mathcal H, \mathcal K)$, recall that every positive operator
has a unique positive square root and that $|T| := \sqrt{T^*T}.$

\begin{thm}\label{!!}
Let $\mathcal{H}$ and $\mathcal{K}$ be complex Hilbert spaces and $T\in \mathcal{B}(\mathcal{H},\mathcal{K})$. Then the following statements are equivalent.\\
$(1)$ $T$ satisfies $\mathcal{N}$.\\
$(2)$ $T^*$ satisfies $\mathcal{N}$.\\
$(3)$ $\||T|\|$ is an eigenvalue of $|T|$.\\
$(4)$ $\|T\|$ is an eigenvalue of $|T|.$\\
$(5)$ $|T|$ satisfies $\mathcal{N}$.\\
$(6)$ $|T^*|$ satisfies $\mathcal{N}$.\\
$(7)$ $|T|^2$ satisfies $\mathcal{N}$.\\
$(8)$ $|T^*|^2$ satisfies $\mathcal{N}$.\\
$(9)$ $\|T\|$ is an eigenvalue of $|T^*|$.
\end{thm} 
\begin{proof} The equivalence of (1) and (7) follows from Theorem~\ref{TiffT*T} as does the equivalence of (5) and (7). Since $\| |T|\|= \|T\|,$ by Theorem~\ref{!}, (5) is equivalent to (3) and (4). 
 
Replacing $T$ by $T^*$ in these equivalences and using that $\|T\|=\|T^*\|$, shows the equivalence of (2), (6), (8) and (9). 

All that remains is to show the equivalence of $(1)$ and $(2)$.
 Assume that $T$ satisfies $\mathcal{N}$. By equivalence of $(1)$ and $(4)$, $\|T\|$ is an eigenvalue of $|T|$. Let $z\in \mathcal{H}$ be an eigenvector of $|T|$ of unit norm corresponding to the eigenvalue $\|T\|$.  Since $|T|(z) = \|T\|z$ we have $T^*Tz=|T|^2(z) = |T|(|T|(z)) = \|T\|^2z .$  Consequently, $\|T^*(Tz)\| =\|T\|^2= \|T\|\|T^*\|$, since $\|z\|=1.$ Notice that $T(\frac{z}{\|T\|})$ is in the unit sphere of $\mathcal{K}$ and hence $\|T^*(T\frac{z}{\|T\|})\| = \|T^*\|$ which means that $T^*$ satisfies $\mathcal{N}.$ This proves that $(1)$ implies $(2)$. The backward implication follows if we replace $T$ by $T^*$ in the proof and use $T^{**}=T.$ This completes the proof.  
\end{proof}
 \begin{remark} Later (see \ref{coisom}) we will give an example of an operator such that $T$ is $\mathcal{AN}$ but $T^*$ is not $\mathcal{AN}.$
 \end{remark}

\section{Necessary conditions for positive $\mathcal{AN}$ Operators}

The purpose of this section is to study the properties of positive $\mathcal{AN}$ operators.

\begin{thm}
Let $T\in \mathcal{B}(\mathcal{H})$ be a positive $\mathcal{AN}$ operator. Then $\mathcal{H}$
has an orthonormal basis consisting of eigenvectors of $T$.
\end{thm}

\begin{proof}
Let $\mathcal{B} = \{v_\alpha: \alpha\in \Lambda\}$ be the maximal orthonormal set of eigenvectors of $T$. That $\mathcal{B}$ is nonempty is a trivial observation; for $T$, being a positive $\mathcal{AN}$ operator, must have $\|T\|$ as its eigenvalue. Considering $w$ to be a unit eigenvector corresponding to the eigenvalue $\|T\|$, we have $Tw = \|T\|w$ which implies that $w\in \mathcal{B}.$

To show that $\mathcal{H}$ has an orthonormal basis consisting entirely of eigenvectors of $T$ we define $\mathcal{H}_0 := \text{clos}(\text{span}(\mathcal{B}))$ and show that $\mathcal{H}_0 = \mathcal{H}.$ It suffices to show that $\mathcal{H}_0^{\perp} = \{0\}$; for then $\mathcal{H}_0 = \mathcal{H}_0^{\perp \perp} = \{0\}^{\perp} = \mathcal{H}.$

We first claim that $\mathcal{H}_0^{\perp}$ is an invariant subspace of $\mathcal{H}$ under $T$. To see this, let 
$\mathcal{F}$ denote the collection of finite subsets of $\Lambda$, that is, $\mathcal{F} = \{F\subseteq \Lambda: F \text{ is finite}\}$. If $v\in \mathcal{H}_0$, then
$$v=\sum_{\alpha\in \Lambda}\inner{v}{v_\alpha}v_\alpha 
= \lim_{F\in \mathcal{F}} \sum_{\alpha \in F}\inner{v}{v_\alpha}v_\alpha.$$
Since 
the above limit is the norm limit and 
$T$ is bounded (norm continuous), it follows that 
$$
Tv 
= T\Big(\lim_{F\in \mathcal{F}} \sum_{\alpha \in F}\inner{v}{v_\alpha}v_\alpha\Big)
= \lim_{F\in \mathcal{F}}\sum_{\alpha\in F}\inner{v}{v_\alpha}\beta_\alpha v_\alpha
= \sum_{\alpha\in \Lambda}\inner{v}{v_\alpha}\beta_\alpha v_\alpha \in \mathcal{H}_0,
$$
considering $Tv_\alpha= \beta_\alpha v_\alpha$ where $\beta_\alpha \in \mathbb{C}$ for every $\alpha\in \Lambda.$ This shows that $\mathcal{H}_0$ is an invariant subspace of $\mathcal{H}$ under $T$. Since $T=T^*$, we infer that $\mathcal{H}_0^{\perp}$ is also an invariant subspace of $\mathcal{H}$ under $T$.

We complete the proof by showing $\mathcal{H}_0^{\perp} = \{0\}$. Suppose, on the contrary, that $\mathcal{H}_0^{\perp} \neq \{0\}$, that is, $\mathcal{H}_0^{\perp}$ is a nontrivial closed subspace of $\mathcal{H}$.  Since $T$ is a positive $\mathcal{AN}$ operator, $T|_{\mathcal{H}_0^{\perp}}$ satisfies the  property $\mathcal{N}$. Even more, $T|_{\mathcal{H}_0^{\perp}}$ is a positive operator on $\mathcal{H}_0^{\perp}$ which satisfies $\mathcal{N}$ because $\mathcal{H}_0^{\perp}$ is invariant under $T$. Consequently, $\|T|_{\mathcal{H}_0^{\perp}}\|$ is an eigenvalue of $T|_{\mathcal{H}_0^{\perp}}$. Let $z$ be a unit eigenvector of $T|_{\mathcal{H}_0^{\perp}}$ corresponding to the eigenvalue $\|T|_{\mathcal{H}_0^{\perp}}\|$. Clearly then $z\in \mathcal{H}_0^{\perp}$ such that $\|z\|=1$ and $T|_{\mathcal{H}_0^{\perp}}(z) = \|T|_{\mathcal{H}_0^{\perp}}\|z$, which implies that $Tz = T|_{\mathcal{H}_0^{\perp}}(z) = \|T|_{\mathcal{H}_0^{\perp}}\|z$. But this means that $z\notin \mathcal{H}_0$ is an eigenvector of $T$ which contradicts the maximality of the set $\mathcal{B} = \{v_\alpha:\alpha\in \Lambda\}$ of $T$ and we conclude that $\mathcal{H}_0^{\perp} = \{0\}.$ This completes the proof.
\end{proof} 
Let $\mathcal{H}$, $\mathcal{K}$ be Hilbert spaces with $w\in \mathcal{H}$ and $v\in \mathcal{K}$. $v\otimes w$ then denotes the operator from $\mathcal{H}$ to $\mathcal{K}$ defined as:
$(v\otimes w)x= \inner{x}{w}v$ for every $x\in \mathcal{H}$.

\begin{cor}\label{ST}
If $T\in \mathcal{B}(\mathcal{H})$ is a positive $\mathcal{AN}$ operator, then 
$$
T = \sum_{\alpha \in \Lambda}\beta_\alpha v_\alpha \otimes v_\alpha ,
$$
where $\{v_\alpha: \alpha\in \Lambda\}$ is an orthonormal basis consisting entirely of eigenvectors of $T$ and for every $\alpha \in \Lambda$, $Tv_\alpha = \beta_\alpha v_\alpha$ with $\beta_\alpha \geqslant 0.$ Moreover, for every nonempty subset $\Gamma\subseteq\Lambda$ of $\Lambda$, we have $\sup\{\beta_\alpha:\alpha\in \Gamma\}=\max\{\beta_\alpha:\alpha\in \Gamma\}.$
\end{cor}

\begin{proof}
That $T = \sum_{\alpha \in \Lambda}\beta_\alpha v_\alpha \otimes v_\alpha$
is obvious. 
To prove the final claim we use the method of contradiction and assume on the contrary that 
$\sup\{\beta_\alpha:\alpha\in \Gamma\}\neq\max\{\beta_\alpha:\alpha\in \Gamma\}$ for some nonempty subset $\Gamma\subseteq\Lambda$, that is, the supremum of the set $\{\beta_\alpha:\alpha\in \Gamma\}$ (say $\beta$) is not achieved. In that case, for any $x\in \mathcal{H}_\Gamma$ with $\|x\|=1$,
$$
\|T|_{\mathcal{H}_\Gamma}(x)\|^2 
                                          = \sum_{\alpha \in \Gamma}|\beta_\alpha|^2 |\inner{x}{v_\alpha}|^2
                                          < \sum_{\alpha \in \Gamma}\beta^2 |\inner{x}{v_\alpha}|^2
                                          = \beta^2 =  \|T|_{\mathcal{H}_\Gamma}\|^2.                                   $$
This implies that $\|T|_{\mathcal{H}_\Gamma}(x)\|<\|T|_{\mathcal{H}_\Gamma}\|$ for every $x\in \mathcal{H}_\Gamma$ with $\|x\|=1$ which contradicts the fact that $T$ is an $\mathcal{AN}$ operator. This proves the assertion.
\end{proof}

The spectral conditions given in the above corollary do not characterize $\mathcal{AN}$ operators as the following example and result show.

\begin{exam}
Let $K_1,K_2$ be positive compact operators that are not of finite rank on the complex Hilbert space $l^2$, and $0\leqslant a<b.$ Consider the operator 
$$
T=\begin{bmatrix}
 aI+K_1 & 0 \\
  0 & bI+K_2
\end{bmatrix} \in \mathcal{B}(l^2\oplus l^2).
$$
Then the supremum of each subset of the spectrum is equal to the maximum of that subset since the spectrum of $T$ consists of the closure of the union of two decreasing sequences, $\{ a_n \}$ and $\{ b_n \}$, with $\lim_n a_n =a$ and $\lim_n b_n =b.$
However, the spectrum of $T$ has two limit points, and so by the following result $T$ is not $\mathcal{AN}.$ Thus, the spectral condition given by the above corollary does not characterize positive $\mathcal{AN}$ operators.
\end{exam}

\begin{prop}
If $T\in \mathcal{B}(\mathcal{H})$ is a positive $\mathcal{AN}$ operator, then the spectrum $\sigma(T)$ of $T$ has at most one limit point. Moreover, this unique limit point (if it exists) can only be the limit of a decreasing sequence in the spectrum.
\end{prop}

\begin{proof}
By the Corollary \ref{ST},  we know that 
$$
T = \sum_{\alpha \in \Lambda}\beta_\alpha v_\alpha \otimes v_\alpha,
$$
where $\{v_\alpha: \alpha\in \Lambda\}$ is an orthonormal basis consisting entirely of eigenvectors of $T$ and for every $\alpha \in \Lambda, Tv_\alpha = \beta_\alpha v_\alpha$ with $\beta_\alpha \geqslant 0$. All that remains is to show that  the spectrum $\sigma(T)$, which is precisely the closure of $\{\beta_\alpha\}_{\alpha\in \Lambda}$ has at most one limit point and this unique limit point (if it exists) can only be the limit of a decreasing sequence in the spectrum.

 First we show that whenever $\lambda$ is a limit point of the spectrum $\sigma(T)$ of $T$, then there exists a decreasing sequence $(\lambda_n)_{n\in \mathbb{N}}\subseteq \{\beta_\alpha:\alpha \in \Lambda\}$ such that $\lambda_n\searrow \lambda$. 
To see this, it is sufficient to prove that there are at most only finitely many terms of the sequence of $(\lambda_n)_{n\in \mathbb{N}}$ that are strictly less than $\lambda$;
for if there are infinitely many such terms, then there exists an increasing subsequence  $(\lambda_{n_k})$ such that 
$\lambda_{n_k} \nearrow \lambda$ and for each $n_k\in \mathbb{N}$, $\lambda_{n_k}<\lambda$.  But then if we define $\mathcal{M}_0:=\text{clos}(\text{span}\{v_{n_k}\})$, where $v_{n_k}$'s are the eigenvectors corresponding to 
the eigenvalues $\lambda_{n_k},$ then it is a trivial observation that 
$\|T|_{\mathcal{M}_0}\| = \sup\{|\lambda_{n_k}|\}=\lambda.$ However, for every 
$x= \sum_{n_k}\alpha_{n_k}v_{n_k}\in \mathcal{M}_0$ with 
$\sum_{n_k}|\alpha_{n_k}|^2=1$ so that $\|x\|=1$, 
$$
\|T|_{\mathcal{M}_0}(x)\|^2 = \|\sum_{n_k}\alpha_{n_k}\lambda_{n_k}v_{n_k}\|^2=
\sum_{n_k}|\alpha_{n_k}|^2|\lambda_{n_k}|^2<\lambda^2 \sum_{n_k}|\alpha_{n_k}|^2 =\lambda^2,
$$ 
so that  $\|T|_{\mathcal{M}_0}(x)\| < \lambda \leqslant \|T|_{\mathcal{M}_0}\|$. This contradicts the 
fact that $T$ is an $\mathcal{AN}$ operator. This proves our first claim.

We next prove, by the method of contradiction, that the spectrum $\sigma(T)$ of $T$ has at most one limit point. Suppose on the contrary that the spectrum 
$\sigma(T)= \text{clos}(\{\beta_\alpha\}_{\alpha\in \Lambda})$ has two limit points 
$a<b.$ By the discussion in the above paragraph, there exist decreasing sequences 
$(a_n)_{n\in \mathbb{N}}\subseteq \{\beta_\alpha\}_{\alpha\in \Lambda}$ and $(b_n)_{n\in \mathbb{N}}\subseteq\{\beta_\alpha\}_{\alpha\in \Lambda}$ such that 
$a_n\searrow a$ and $b_n\searrow b.$ Let us rename and denote by $\{f_n\}$ and $\{g_n\}$ the eigenvectors corresponding to the eigenvalues $\{a_n\}$ and $\{b_n\}$ respectively. Without any loss of generality we may assume that $a_1<b$ so that $a_n<b_n$ for each $n\in \mathbb{N}$. (For if it happens otherwise then we can choose a natural number $m$ such that $a_m<b$ and redefine the sequence $(a_n)_{n=m}^\infty$ by $(\tilde{a}_n)_{n=1}^\infty$.) Also note that 
$Tf_n = a_nf_n$ and $Tg_n=b_ng_n$ for each $n\in \mathbb{N}.$
Define 
$$
\mathcal{M}:=\text{clos}\left(\text{span}\left\{c_nf_n + \sqrt{1-c_n^2}g_n: n\in \mathbb{N}\right\}\right),
$$
where $c_n^2\in [0,1]$ are yet to be determined.
Needless to say that $\mathcal{M}$ is a closed subspace of $\mathcal{H}$ and hence a Hilbert space in its own right. Moreover, it is a trivial observation that
the set $\left\{e_n: n\in \mathbb{N}\right\}$, where $e_n:=c_nf_n + \sqrt{1-c_n^2}g_n$ serves to be an orthonormal basis of $\mathcal{M}.$
Then,
\begin{multline*}
\|T|_{\mathcal{M}}\|^2 \geqslant \sup\{\|Te_n\|^2\}
  =\sup\{\|T(c_nf_n + \sqrt{1-c_n^2}g_n)\|^2:n\in \mathbb{N}\}\\
  =\sup\{\|c_na_nf_n + \sqrt{1-c_n^2}b_ng_n\|^2:n\in \mathbb{N}\}
  =\sup\{c_n^2a_n^2 + (1-c_n^2)b_n^2: n\in \mathbb{N}\}.
\end{multline*}
At this point we define a sequence $(\gamma_n)_{n\in \mathbb{N}}$ by 
$$
\gamma_n:=b+\frac{a_1-b}{2n}; n\in \mathbb{N}. 
$$
Then, $(\gamma_n)_{n\in \mathbb{N}}$ is a strictly increasing sequence such that for every $n\in \mathbb{N}$, $a_1^2<\gamma_n^2<b^2$ and 
$\lim_{n\rightarrow \infty} \gamma_n =\sup\{\gamma_n:n\in \mathbb{N}\}= b.$ Notice that $c_n^2a_n^2 + (1-c_n^2)b_n^2$ is a convex combination of $a_n^2$ and $b_n^2$, and hence it follows that $ c_n^2a_n^2 + (1-c_n^2)b_n^2 \in [a_n^2,b_n^2]$ for each $n\in \mathbb{N}$. In fact, by choosing the right value of $c_n^2\in [0,1], c_n^2a_n^2 + (1-c_n^2)b_n^2$ can give any point in the interval $[a_n^2,b_n^2]$. Let us then choose a sequence $(c_n)_{n\in \mathbb{N}}$ such that $c_n^2a_n^2 + (1-c_n^2)b_n^2=\gamma_n^2$.
This yields
$$
\|T|_{\mathcal{M}}\|^2 \geqslant \sup\{c_n^2a_n^2 + (1-c_n^2)b_n^2: n\in \mathbb{N}\}
 = \sup\{\gamma_n^2:n\in \mathbb{N}\} = b^2.
$$
However, any $x\in \mathcal{M}$ with $\|x\|=1$ can be written as 
$$
\sum_{n=1}^\infty \alpha_n (c_nf_n+ \sqrt{1-c_n^2}g_n), \text{ with } \sum_{n=1}^\infty|\alpha_n|^2=1,
$$
in which case,
\begin{multline*}
\|T|_{\mathcal{M}}(x)\|^2 =\|Tx\|^2
                             =\left\|T\left(\sum_{n=1}^\infty \alpha_n (c_nf_n+ \sqrt{1-c_n^2}g_n)\right)\right\|^2\\
                             = \sum_{n=1}^\infty |\alpha_n|^2(c_n^2a_n^2 + (1-c_n^2)b_n^2)
                             = \sum_{n=1}^\infty |\alpha_n|^2\gamma_n^2
                              < \sum_{n=1}^\infty |\alpha_n|^2b^2
                             =b^2.
\end{multline*} 
This implies that for every element $x\in \mathcal{M}$ with $\|x\|=1$, $\|T|_\mathcal{M}(x)\| < b \leqslant \|T|_\mathcal{M}\|$ which means that $T$ does not satisfy the property $\mathcal{AN}.$ So we arrive at a contradiction. Hence, our hypothesis is wrong and we conclude that the spectrum of $T$ can have at most one limit point. This completes the proof.
\end{proof}
We now use this as a tool to prove the following result.
\begin{cor}
If $T\in \mathcal{B}(\mathcal{H})$ is a positive $\mathcal{AN}$ operator, then the set $\{\beta_\alpha\}_{\alpha\in \Lambda}$ of distinct eigenvalues of $T$, that is, without counting multiplicities, is countable.
\end{cor}
\begin{proof}
This corollary is a direct consequence of the fact that 
\textit{If $E\subseteq \mathbb{R}$ is an uncountable subset, then $E$ has at least two limit points.} Since the set $\{\beta_\alpha\}_{\alpha\in \Lambda}$
has at most one limit point, by the contrapositive of the above fact, it is countable.
\end{proof}

\begin{cor}\label{oneinfmult}
If $T\in \mathcal{B}(\mathcal{H})$ is a positive $\mathcal{AN}$ operator, then the set $\{\beta_\alpha\}_{\alpha\in \Lambda}$ of eigenvalues of $T$ has at most one eigenvalue with infinite multiplicity.
\end{cor}
\begin{proof}
To show that this set has at most one eigenvalue with infinite multiplicity, we assume that it has two distinct eigenvalues $\beta_1$ and $\beta_2$ with infinite multiplicity, and we deduce a contradiction from the assumption. Without loss of generality, we assume that $0\leq\beta_1<\beta_2.$ Now let $(a_n)_{n\in \mathbb{N}}\subseteq \{\beta_\alpha\}_{\alpha\in \Lambda}$  and $(b_n)_{n\in \mathbb{N}}\subseteq\{\beta_\alpha\}_{\alpha\in \Lambda}$  be two sequences such that for every $n\in \mathbb{N}$, we have $a_n=\beta_1$ and $b_n=\beta_2$. Clearly then $a_n\longrightarrow \beta_1$ and $b_n\longrightarrow \beta_2.$ Let us, like in the previous proof, rename and denote by $\{f_n\}$ and $\{g_n\}$ the eigenvectors corresponding to the eigenvalues $\{a_n\}$ and $\{b_n\}$ respectively where $Tf_n = a_nf_n= \beta_1 f_n$ and $Tg_n=b_ng_n=\beta_2 g_n$ for each $n\in \mathbb{N}.$
At this point we define a sequence $(\gamma_n)_{n\in \mathbb{N}}$ by
$$
\gamma_n:=\beta_2+\frac{\beta_1-\beta_2}{2n}; n\in \mathbb{N}.
$$
That $(\gamma_n)_{n\in \mathbb{N}}$ is a strictly increasing sequence with $\beta_1^2<\gamma_n^2<\beta_2^2$ for every $n\in \mathbb{N}$ such that $\lim_{n\rightarrow \infty} \gamma_n =\sup\{\gamma_n:n\in \mathbb{N} \}= \beta_2$ is obvious. Let $c_n^2\in [0,1]$ be arbitrary, then since $c_n^2\beta_1^2 + (1-c_n^2)\beta_2^2$ is a convex linear combination of $\beta_1^2$ and $\beta_2^2$, it follows that for each $n\in \mathbb{N}$, we have 
$c_n^2\beta_1^2 + (1-c_n^2)\beta_2^2\in [\beta_1^2,\beta_2^2]$. In fact, by choosing the right value of $c_n^2\in [0,1]$, $c_n^2\beta_1^2 + (1-c_n^2)\beta_2^2$ gives any desired point in  the interval $[\beta_1^2,\beta_2^2]$. 
This observation, together with the fact that $\beta_1^2<\gamma_n^2<\beta_2^2$ for every $n\in \mathbb{N}$, allows us to define the sequence $(c_n)_{n\in \mathbb{N}}$ concretely, which is as follows:\\ for each $n\in \mathbb{N}$, choose $c_n$ so that 
$c_n^2\beta_1^2 + (1-c_n^2)\beta_2^2=\gamma_n^2$.\\
We will use this so defined sequence $(c_n)_{n\in \mathbb{N}}$ as a tool to define a closed subspace $\mathcal{M}$ of $\mathcal{H}$ by 
$$
\mathcal{M}:=\text{clos}\left(\text{span}\left\{c_nf_n + \sqrt{1-c_n^2}g_n: n\in \mathbb{N}\right\}\right).
$$
It is easy to see that the set $\left\{e_n: n\in \mathbb{N}\right\}$ serves to be an orthonormal basis of $\mathcal{M}$, where $e_n:=c_nf_n + \sqrt{1-c_n^2}g_n$. It now follows that
\begin{multline*}
\|T|_{\mathcal{M}}\|^2 \geqslant \sup\{\|Te_n\|^2\}
  =\sup\{\|c_n\beta_1f_n + \sqrt{1-c_n^2}\beta_2g_n\|^2:n\in \mathbb{N}\}\\
  =\sup\{c_n^2\beta_1^2 + (1-c_n^2)\beta_2^2: n\in \mathbb{N}\}
  =\sup\{\gamma_n^2:n\in \mathbb{N}\} = \beta_2^2.
\end{multline*}
However, any $x\in \mathcal{M}$ with $\|x\|=1$ can be written as 
$$
\sum_{n=1}^\infty \alpha_n (c_nf_n+ \sqrt{1-c_n^2}g_n) \text{ with } \sum_{n=1}^\infty|\alpha_n|^2=1.
$$
In that case,
\begin{align*}
\|T|_{\mathcal{M}}(x)\|^2 &=\|Tx\|^2=\left\|T\left(\sum_{n=1}^\infty \alpha_n (c_nf_n+ \sqrt{1-c_n^2}g_n)\right)\right\|^2\\
                              &=\left\|\sum_{n=1}^\infty \alpha_n (c_n\beta_1f_n+ \sqrt{1-c_n^2}\beta_2g_n)\right\|^2\\ 
                             &= \sum_{n=1}^\infty |\alpha_n|^2(c_n^2\beta_1^2 + (1-c_n^2)\beta_2^2)\\
                             &= \sum_{n=1}^\infty |\alpha_n|^2\gamma_n^2 < \sum_{n=1}^\infty |\alpha_n|^2\beta_2^2 =\beta_2^2.
\end{align*} 
This implies that for every element $x\in \mathcal{M}$ with $\|x\|=1$, $\|T|_\mathcal{M}(x)\| < \beta_2 \leqslant \|T|_\mathcal{M}\|$ which means that $T$ does not satisfy the property $\mathcal{AN}.$ So we arrive at a contradiction. Hence, our hypothesis was wrong and we conclude that the spectrum of $T$ can have at most one eigenvalue with infinite multiplicity. This completes the proof.
\end{proof}

\begin{cor}
Let $T\in \mathcal{B}(\mathcal{H})$ be a positive $\mathcal{AN}$ operator. If the spectrum $\sigma(T)=\text{clos}\{\beta_\alpha:\alpha\in \Lambda\}$ of $T$ has both a limit point $\beta$ and an eigenvalue $\hat \beta $ with infinite multiplicity, then $\beta=\hat \beta.$
\end{cor}

\begin{proof}
To show that $\beta=\hat \beta$, we assume that $\beta\neq\hat \beta$, and we deduce a contradiction from the assumption. We first consider the case when $\beta<\hat \beta$. Because $\beta$ is a limit point of the spectrum, we know that there exists a decreasing sequence $(a_n)_{n\in \mathbb{N}}\subseteq \{\beta_\alpha\}_{\alpha\in \Lambda}$ such that $a_n\searrow \beta$. Let $(b_n)_{n\in \mathbb{N}}\subseteq\{\beta_\alpha\}_{\alpha\in \Lambda}$  be the constant sequence whose each term is $\hat\beta $ so that $b_n\longrightarrow \hat \beta.$  Without any loss of generality we may assume that $a_1<\hat\beta$ so that $a_n<b_n$ for each $n\in \mathbb{N}$. Next we rename and denote by $\{f_n\}$ and $\{g_n\}$ the eigenvectors corresponding to the eigenvalues $\{a_n\}$ and $\{b_n\}$ respectively where $Tf_n = a_nf_n$ and $Tg_n=b_ng_n=\hat\beta g_n$ for each $n\in \mathbb{N}.$

As we did in the previous proof, we define a sequence $(\gamma_n)_n\in \mathbb{N}$ by
$$
\gamma_n:=\hat \beta + \frac{\beta-\hat\beta}{2n}; \text{  }n\in \mathbb{N}.
$$
Observe that $(\gamma_n)_{n\in \mathbb{N}}$ is a strictly increasing sequence with $a_n^2<\gamma_n^2<\hat\beta^2$ for every $n\in \mathbb{N}$. It immediately follows then that $\lim_{n\rightarrow \infty}\gamma_n=\sup\{\gamma_n:n\in \mathbb{N}\}= \hat \beta$. Thereafter for each $n\in \mathbb{N}$, we choose $c_n$ so that $c_n^2\in [0,1]$ and $c_n^2a_n^2 +(1-c_n^2)\hat \beta =\gamma_n^2$. Finally, with the help of this sequence $(c_n)_{n\in \mathbb{N}}$ let us define a closed subspace $\mathcal{M}$ of $\mathcal{H}$ by 
$$
\mathcal{M}:=\text{clos}\left(\text{span}\left\{c_nf_n + \sqrt{1-c_n^2}g_n: n\in \mathbb{N}\right\}\right).
$$
We know that the set $\left\{e_n: n\in \mathbb{N}\right\}$, where $e_n:=c_nf_n + \sqrt{1-c_n^2}g_n$, is an orthonormal basis of $\mathcal{M}.$ It now follows, like the argument in the previous proof, that
\begin{multline*}
\|T|_{\mathcal{M}}\|^2 \geqslant \sup\{\|Te_n\|^2\}=\sup\{\|c_na_nf_n + \sqrt{1-c_n^2}\hat \beta g_n\|^2:n\in \mathbb{N}\}\\
  =\sup\{c_n^2a_n^2 + (1-c_n^2)\hat\beta^2: n\in \mathbb{N}\}
  =\sup\{\gamma_n^2:n\in \mathbb{N}\} = \hat\beta^2.
\end{multline*}
Since each $x\in \mathcal{M}$ with $\|x\|=1$ can be written as 
$$
\sum_{n=1}^\infty \alpha_n (c_nf_n+ \sqrt{1-c_n^2}g_n) \text{ with } \sum_{n=1}^\infty|\alpha_n|^2=1, \text{ it follows that }
$$
\begin{align*}
\|T|_{\mathcal{M}}(x)\|^2 &=\|Tx\|^2=\left\|T\left(\sum_{n=1}^\infty \alpha_n (c_nf_n+ \sqrt{1-c_n^2}g_n)\right)\right\|^2\\
                              &=\left\|\sum_{n=1}^\infty \alpha_n (c_na_nf_n+ \sqrt{1-c_n^2}\hat \beta g_n)\right\|^2\\ 
                             &= \sum_{n=1}^\infty |\alpha_n|^2(c_n^2a_n^2 + (1-c_n^2)\hat\beta ^2)\\
                             &= \sum_{n=1}^\infty |\alpha_n|^2\gamma_n^2 < \sum_{n=1}^\infty |\alpha_n|^2\hat\beta^2 =\hat\beta^2.
\end{align*} 
This implies that for every element $x\in \mathcal{M}$ with $\|x\|=1$, $\|T|_\mathcal{M}(x)\| < \hat\beta \leqslant \|T|_\mathcal{M}\|$ which means that $T$ does not satisfy the property $\mathcal{AN}.$ So we arrive at a contradiction. Hence, our hypothesis was wrong and we conclude that $\beta=\hat \beta$.

To prove the assertion for the case when $\hat \beta < \beta$, we follow the same line of argument. Let  $(a_n)_{n\in \mathbb{N}}\subseteq \{\beta_\alpha\}_{\alpha\in \Lambda}$  be the decreasing sequence such that $a_n\searrow \beta$, $(b_n)_{n\in \mathbb{N}}\subseteq\{\beta_\alpha\}_{\alpha\in \Lambda}$  be the constant sequence whose each term is $\hat\beta $ so that $b_n\longrightarrow \hat \beta$, and rename and denote by $\{f_n\}$ and $\{g_n\}$ the eigenvectors corresponding to the eigenvalues $\{a_n\}$ and $\{b_n\}$ respectively where $Tf_n = a_nf_n$ and $Tg_n=b_ng_n=\hat\beta g_n$ for each $n\in \mathbb{N}.$
We define the sequence $(\gamma_n)_n\in \mathbb{N}$ a bit differently by
$$
\gamma_n:= \beta + \frac{\hat\beta-\beta}{2n}; n\in \mathbb{N}.
$$
It is now a trivial observation that $(\gamma_n)_{n\in \mathbb{N}}$ is a strictly increasing sequence with $\hat\beta^2<\gamma_n^2<a_n^2$ for every $n\in \mathbb{N}$. Consequently, $\lim_{n\rightarrow \infty}\gamma_n=\sup\{\gamma_n:n\in \mathbb{N}\}= \beta$.

 Thereafter for each $n\in \mathbb{N}$, we choose $c_n$ so that $c_n^2\in [0,1]$ and $c_n^2\hat \beta^2 +(1-c_n^2)a_n^2 =\gamma_n^2$. Finally, with the help of this sequence $(c_n)_{n\in \mathbb{N}}$, we define a closed subspace $\mathcal{\hat M}$ of $\mathcal{H}$ by 
$$
\mathcal{\hat M}:=\text{clos}\left(\text{span}\left\{c_ng_n + \sqrt{1-c_n^2}f_n: n\in \mathbb{N}\right\}\right).
$$
That the set $\left\{e_n: n\in \mathbb{N}\right\}$, where $e_n:=c_ng_n + \sqrt{1-c_n^2}gf_n$, is an orthonormal basis of $\mathcal{\hat M}$ can be easily verified. It now follows that
\begin{multline*}
\|T|_{\mathcal{\hat M}}\|^2 \geqslant \sup\{\|Te_n\|^2\}=\sup\{\|c_n\hat \beta g_n + \sqrt{1-c_n^2} a_nf_n\|^2:n\in \mathbb{N}\}\\
  =\sup\{c_n^2\hat \beta^2 + (1-c_n^2)a_n^2: n\in \mathbb{N}\}
  =\sup\{\gamma_n^2:n\in \mathbb{N}\} =\beta^2.
\end{multline*}
Since each $x\in \mathcal{\hat M}$ with $\|x\|=1$ can be written as 
$$
\sum_{n=1}^\infty \alpha_n (c_ng_n+ \sqrt{1-c_n^2}f_n) \text{ with } \sum_{n=1}^\infty|\alpha_n|^2=1, \text { it follows that }
$$
\begin{align*}
\|T|_{\mathcal{\hat M}}(x)\|^2 &=\|Tx\|^2=\left\|T\left(\sum_{n=1}^\infty \alpha_n (c_ng_n+ \sqrt{1-c_n^2}f_n)\right)\right\|^2\\
                             &= \sum_{n=1}^\infty |\alpha_n|^2(c_n^2\hat \beta^2 + (1-c_n^2)  a_n^2)\\
                             &= \sum_{n=1}^\infty |\alpha_n|^2\gamma_n^2 < \sum_{n=1}^\infty |\alpha_n|^2\beta^2 =\beta^2.
\end{align*} 
This implies that for every element $x\in \mathcal{\hat M}$ with $\|x\|=1$, $\|T|_\mathcal{\hat M}(x)\| < \beta \leqslant \|T|_\mathcal{\hat M}\|$ which contradicts the fact that $T$ satisfies the property $\mathcal{AN}$. Thus, we conclude that $\beta=\hat \beta$. This completes the proof.
\end{proof}

We finish this section by stating the final proposition in its full strength.

\begin{thm}\label{NC}
If $T\in \mathcal{B}(\mathcal{H})$ is a positive $\mathcal{AN}$ operator, then 
$$
T = \sum_{\alpha \in \Lambda}\beta_\alpha v_\alpha \otimes v_\alpha,
$$
where $\{v_\alpha: \alpha\in \Lambda\}$ is an orthonormal basis consisting entirely of eigenvectors of $T$ and for every $\alpha \in \Lambda, Tv_\alpha = \beta_\alpha v_\alpha$ with $\beta_\alpha \geqslant 0$ such that \\

$(i)$ for every nonempty subset $\Gamma \subseteq \Lambda$ of $\Lambda$, we have  $\sup\{\beta_\alpha :\alpha \in \Gamma\}=\text{max}\{\beta_\alpha :\alpha \in \Gamma\};$\\

$(ii)$ the spectrum $\sigma(T)=\text{clos }\{\beta_\alpha:\alpha\in \Lambda\}$ of $T$ has at most one limit point. Moreover, this unique limit point (if it exists) can only be the limit of a decreasing sequence in the spectrum;\\

$(iii)$ the set $\{\beta_\alpha\}_{\alpha\in \Lambda}$ of eigenvalues of $T$, without counting multiplicities, is countable and has at most one eigenvalue with infinite multiplicity;\\

$(iv)$ if the spectrum $\sigma(T)=\text{clos}\{\beta_\alpha:\alpha\in \Lambda\}$ of $T$ has both, a limit point $\beta$ and an eigenvalue $\hat \beta $ with infinite multiplicity, then $\beta=\hat \beta.$
\end{thm}

\section{Sufficient Conditions For $\mathcal{AN}$ Operators}
We now discuss the sufficient conditions for an operator (not necessarily positive) to satisfy the $\mathcal{AN}$ condition.
There is an important and useful criterion for an operator $T\in \mathcal{B}(\mathcal{H},\mathcal{K})$
to satisfy the property $\mathcal{AN}$, which depends on the following facts: For a closed linear subspace $\mathcal{M}$ of a complex Hilbert space $\mathcal{H}$ let $V_{\mathcal{M}}:\mathcal{M}\longrightarrow \mathcal{H}$ be the inclusion map from $\mathcal{M}$ to $\mathcal{H}$ defined as $V_{\mathcal{M}}(x) = x$ for each $x\in \mathcal{M}$. It is then a trivial observation that the adjoint $V_\mathcal{M}^*:\mathcal{H}\longrightarrow \mathcal{M}$ of $V_\mathcal{M}$ is the orthogonal projection of $\mathcal{H}$ on $\mathcal{M}$ (viewed as a map from $\mathcal{H}$ onto $\mathcal{M}$), that is, $V_\mathcal{M}^*:\mathcal{H}\longrightarrow \mathcal{M}$ such that 
$$
 V_M^*(y)= \begin{cases} 
     y & \text{ if } y\in M, \\
      0 & \text{ if } y\in M^\perp.\\
   \end{cases} 
$$
The criterion referred to is the following: $T$ satisfies the property $\mathcal{AN}$ iff  for every closed linear subspace $\mathcal{M}$ of $\mathcal{H}$, $TV_{\mathcal{M}}$ satisfies the property $\mathcal{N}$.
To prove this assertion we first observe that for any given nontrivial closed subspace $\mathcal{M}$ of $\mathcal{H}$, $\|TV_\mathcal{M}\| = \|T|_\mathcal{M}\|$; for 
\begin{align*}
\|TV_\mathcal{M}\|^2 &=\sup\{\|TV_\mathcal{M}(x)\|^2: \|x\|\leq1, x\in \mathcal{M}\}\\
                & = \sup\{\|Tx\|^2: \|x\|\leq1, x\in \mathcal{M}\}=\|T|_\mathcal{M}\|^2.
\end{align*}
We next assume that $T$ satisfies the property $\mathcal{AN}$ and prove the forward implication. Let $\mathcal{M}$ be an arbitrary nontrivial closed subspace of $\mathcal{H}$. Clearly then there exists $x_0\in \mathcal{M}$ with $\|x_0\|=1$ such that $\|T|_\mathcal{M}\| = \|Tx_0\|$. It follows then that there exists $x_0\in \mathcal{H}$ such that $\|TV_\mathcal{M}\| =\|T|_\mathcal{M}\|=\|Tx_0\| = \|TV_\mathcal{M} (x_0)\|$. Since $\mathcal{M}$ is arbitrary, it follows that $TV_\mathcal{M}$ satisfies the property $\mathcal{N}$.
We complete the proof by showing that $T$ is an $\mathcal{AN}$ operator if $TV_\mathcal{M}$ satisfies the property $\mathcal{N}$ for every nontrivial closed subspace $\mathcal{M}$ of $\mathcal{H}$. Since $TV_\mathcal{M}$ is an $\mathcal{N}$ operator, there exists $x_\mathcal{M}\in \mathcal{H}$(depending on $\mathcal{M}$)  with $\|x_\mathcal{M}\| = 1$ and $\|TV_\mathcal{M}\|=\|TV_\mathcal{M}(x_\mathcal{M})\|$. This means that for every 
$\mathcal{M}$,
$
\|T|_\mathcal{M}\|=\|TV_\mathcal{M}\| = \|TV_\mathcal{M} (x_\mathcal{M})\|=\|T(V_\mathcal{M}x_\mathcal{M})\| =\|Tx_\mathcal{M}\| = \|T|_\mathcal{M}(x_\mathcal{M})\| 
$ 
where $x_\mathcal{M}\in \mathcal{M}$ and $\|x_\mathcal{M}\|=1$.
 This essentially guarantees that for every $\mathcal{M}$, $T|_\mathcal{M}$ achieves its norm on unit sphere and hence satisfies the property $\mathcal{N}$.

We can summarize the result of the above discussion in the following lemma:

\begin{lemma}\label{Trick}
For a closed linear subspace $\mathcal{M}$ of a complex Hilbert space $\mathcal{H}$ let $V_{\mathcal{M}}:\mathcal{M}\longrightarrow \mathcal{H}$ be the inclusion map from $\mathcal{M}$ to $\mathcal{H}$ defined as $V_{\mathcal{M}}(x) = x$ for each $x\in \mathcal{M}$.
 An operator $T\in \mathcal{B}(\mathcal{H},\mathcal{K})$ satisfies the property $\mathcal{AN}$ if and only if for every nontrivial closed linear subspace $\mathcal{M}$ of $\mathcal{H}$, $TV_{\mathcal{M}}$ satisfies the property $\mathcal{N}$.
\end{lemma}

The following application illustrates the power of this result.

\begin{prop}\label{Isometry is AN}
If $T\in \mathcal{B}(\mathcal{H},\mathcal{K})$ is an isometry, then $T$ satisfies the property $\mathcal{AN}$.
\end{prop}
\begin{proof}
That an isometry satisfies the property $\mathcal{N}$ is obvious; for the operator norm of an isometry is $1$ and it attains its norm on any vector of unit length.
For a closed linear subspace $\mathcal{M}$ of the Hilbert space $\mathcal{H}$ let $V_{\mathcal{M}}:\mathcal{M}\longrightarrow \mathcal{H}$ be the inclusion map from $\mathcal{M}$ to $\mathcal{H}$ defined as $V_{\mathcal{M}}(x) = x$ for each $x\in \mathcal{M}$. To prove the assertion, it suffices to show that for every nonzero closed linear subspace $\mathcal{M}$, $TV_{\mathcal{M}}$ is an $\mathcal{N}$ operator. But $TV_{\mathcal{M}}\in \mathcal{B}(\mathcal{M},\mathcal{K})$ is an isometry and hence satisfies the property $\mathcal{N}$. 
\end{proof}

\begin{lemma}\label{SP}
Let $T\in \mathcal{B}(\mathcal{H})$ be a diagonalizable operator on the complex Hilbert space $\mathcal{H}$ and $B=\{v_\alpha:\alpha \in \Lambda\}$ be an orthonormal basis of $\mathcal{H}$ corresponding to which $T$ is diagonalizable. 
If $T$ achieves its norm on the unit sphere of $\mathcal{H}$, then it achieves it norm on some $v_0\in B$.
Alternatively, if $T$ satisfies the property $\mathcal{N}$, then there exists $v_0\in B$ such that $\|T\| = \|Tv_0\|.$
\end{lemma}
\begin{proof}
Let $\{\lambda_\alpha:\alpha\in \Lambda\}$ be the set of eigenvalues of $T$ corresponding to the the eigenvectors $\{v_\alpha:\alpha \in \Lambda\}$. From \cite[Problem 61]{Halmos}, we know that $\|T\| = \sup \{|\lambda_\alpha|:\alpha \in \Lambda\}$, so it suffices to prove that$\|T\| = \max \{|\lambda_\alpha|:\alpha \in \Lambda\}$.

To this end, by the way of contradiction, we assume the negation of the above claim. It implies that for every $\alpha \in \Lambda,$ we have $ |\lambda_\alpha|< \|T\|.$
However, for every $x\in \mathcal{H}$  with $\|x\|=1$, we have $Tx= \sum_{\alpha\in \Lambda} \lambda_\alpha\inner{x}{v_\alpha}v_\alpha$ so that 
$$
\|Tx\|^2= \sum_{\alpha\in \Lambda} |\lambda_\alpha|^2|\inner{x}{v_\alpha}|^2
 < \sum_{\alpha\in \Lambda} \|T\|^2|\inner{x}{v_\alpha}|^2
 =\|T\|^2\|x\|^2
= \|T\|^2;
$$
which is a contradiction of the fact that $T$ satisfies the property $\mathcal{N}$. This proves the claim.
\end{proof}

\begin{lemma}
Let $F\in\mathcal{B}(\mathcal{H})$ be a self-adjoint finite-rank operator and $\alpha \geq 0$. Then $\alpha I+F$ satisfies the property $\mathcal{N}$.
\end{lemma}

\begin{proof}
Let the range of $F$ be $k$-dimensional. Since $F$ is self-adjoint, there exists an orthonormal basis $B = \{v_\lambda:\lambda\in \Lambda\}$ of $\mathcal{H}$ corresponding to which the matrix $M_B(F)$ is a diagonal matrix with $k$ nonzero real diagonal entries, say $\{\beta_1, \beta_2,...,\beta_k\}$. Clearly then, $M_B(\alpha I+F)$ is also a diagonal matrix and 
\begin{align*}
\|\alpha I+F\|&=\sup\{|\alpha+\beta_1|, |\alpha+\beta_2|,...,|\alpha+\beta_k|,\alpha\}\\
          &=\max\{|\alpha+\beta_1|, |\alpha+\beta_2|,...,|\alpha+\beta_k|,\alpha\}.
\end{align*}
It is then a trivial observation that there exists $v_0\in B$ such that $\|\alpha I+F\|=
\|(\alpha I+F)v_0\|$. This proves that $\alpha I+F$ achieves its norm on the unit sphere and hence is an $\mathcal{N}$ operator.
\end{proof}
This lemma leads to the following propostion.
\begin{prop}
If $F\in\mathcal{B}(\mathcal{H})$ is a self-adjoint finite-rank operator and $\alpha \geq 0$, then $\alpha I+F$ satisfies the property $\mathcal{AN}$.
\end{prop}

\begin{proof}
For a closed linear subspace $\mathcal{M}$ of the Hilbert space $\mathcal{H}$ let $V_{\mathcal{M}}:\mathcal{M}\longrightarrow \mathcal{H}$ be the inclusion map from $\mathcal{M}$ to $\mathcal{H}$ defined as $V_{\mathcal{M}}(x) = x$ for each $x\in \mathcal{M}$.

Let us then define $T:= \alpha I+F$ so that we have $T^*=\alpha I+F$ and $T^*T = (\alpha I+F)^2 = \alpha^2 I+2\alpha F+F^2 = \beta I+\tilde F$ where $\beta=\alpha^2 \geq 0$ and $\tilde F = 2\alpha F+F^2$ is another self-adjoint finite-rank operator. We observe that
\begin{align*}
T \text{ is } \mathcal{AN}
&\iff  \text{ for every closed subspace } \mathcal{M} \text{ of } \mathcal{H}, TV_\mathcal{M} \text{ is } \mathcal{N}\\
&\iff  \text{ for every closed subspace } \mathcal{M} \text{ of } \mathcal{H}, (TV_\mathcal{M})^*(TV_\mathcal{M}) \text{ is } \mathcal{N}\\
&\iff  \text{ for every closed subspace } \mathcal{M} \text{ of } \mathcal{H},V_\mathcal{M}^*(T^*T)V_\mathcal{M} \text{ is } \mathcal{N}\\
&\iff  \text{ for every closed subspace } \mathcal{M} \text{ of } \mathcal{H}, V_\mathcal{M}^*(\beta I+\tilde F)V_\mathcal{M} \text{ is } \mathcal{N}.
\end{align*}
So, it suffices to show that for every closed subspace $\mathcal{M}$ of  $\mathcal{H},  V_\mathcal{M}^*(\beta I+\tilde F)V_\mathcal{M}$  is  $\mathcal{N}$. But $V_\mathcal{M}^*(\beta I+\tilde F)V_\mathcal{M}:\mathcal{M}\longrightarrow \mathcal{M}$ is an operator on $\mathcal{M}$ and 
$$
V_\mathcal{M}^*(\beta I+\tilde F)V_\mathcal{M} = V_\mathcal{M}^*\beta IV_\mathcal{M} + V_\mathcal{M}^* \tilde FV_\mathcal{M} = \beta I_\mathcal{M} +\tilde F_\mathcal{M},
$$
which implies that $
V_\mathcal{M}^*(\beta I+\tilde F)V_\mathcal{M}$ is sum of a nonnegative scalar multiple of identity and  a self-adjoint finite-rank operator on a Hilbert space $\mathcal{M}$ which, by the previous lemma, does satisfy the property $\mathcal{N}$ and thus proves our assertion.
\end{proof}

\begin{lemma}
For any positive compact operator $K\in\mathcal{B}(\mathcal{H})$ and $\alpha \geq 0$,  $\alpha I+K$ satisfies the property $\mathcal{N}$.
\end{lemma}

\begin{proof}
That $K$ satisfies the property $\mathcal{N}$ is obvious, for $K$ is compact. The positivity of $K$ ascertains that there is an orthonormal basis $B=\{v_\lambda:\lambda \in \Lambda\}$ of $\mathcal{H}$, consisting entirely of eigenvectors of $K$, corresponding to which $K$ is diagonalizable; this fact , together with the lemma \ref{SP} implies that there exists $v_0\in B$ such that $\|K\| =\beta_0 =\text{max }\{\beta_\lambda:\lambda\in \Lambda\}= \|Kv_0\|,$ where $K(v_\lambda)=\beta_\lambda v_\lambda$ for each $\lambda\in \Lambda$. Since $\alpha\geq 0$, it readily follows  that 
\begin{multline*}
\|\alpha I + K\| = \sup\{\alpha + \beta_\lambda:\lambda\in \Lambda\} = \alpha + \sup\{\beta_\lambda:\lambda\in \Lambda\}  \\= \alpha + \text{max }\{\beta_\lambda:\lambda\in \Lambda\} = \alpha + \beta_0 = \|(\alpha I + K)(v_0)\|.
\end{multline*}
$\alpha I + K$ therefore achieves its norm on unit sphere for each $\alpha\geq 0.$ This completes the proof.
\end{proof}
 
This lemma is a special case of what the following proposition states.

\begin{prop}
For any positive compact operator $K\in\mathcal{B}(\mathcal{H})$ and $\alpha \geq 0$,  $\alpha I+K$ is $\mathcal{AN}$.
\end{prop}

\begin{proof}
Let us define $T:= \alpha I+K$ so that we have $T^*=\alpha I+K$ and $T^*T = (\alpha I+K)^2 = \alpha^2 I+2\alpha K+K^2 = \beta I+\tilde K$ where $\beta=\alpha^2 \geq 0$ and $\tilde K = 2\alpha K+K^2$ is another positive compact operator.
\begin{align*}
T \text{ is } \mathcal{AN}
&\iff  \text{ for every closed subspace } \mathcal{M} \text{ of } \mathcal{H}, TV_\mathcal{M} \text{ is } \mathcal{N}\\
&\iff  \text{ for every closed subspace } \mathcal{M} \text{ of } \mathcal{H}, (TV_\mathcal{M})^*(TV_\mathcal{M}) \text{ is } \mathcal{N}\\
&\iff  \text{ for every closed subspace } \mathcal{M} \text{ of } \mathcal{H}, V_\mathcal{M}^*(T^*T)V_\mathcal{M} \text{ is } \mathcal{N}\\
&\iff  \text{ for every closed subspace } \mathcal{M} \text{ of } \mathcal{H}, V_\mathcal{M}^*(\beta I+\tilde K)V_\mathcal{M} \text{ is } \mathcal{N}.
\end{align*}
So, it suffices to show that for every closed subspace $\mathcal{M}$ of  $\mathcal{H},  V_\mathcal{M}^*(\beta I+\tilde K)V_\mathcal{M}$  is  $\mathcal{N}$. But $V_\mathcal{M}^*(\beta I+\tilde K)V_\mathcal{M}:\mathcal{M}\longrightarrow \mathcal{M}$ is an operator on $\mathcal{M}$ and 
$$
V_\mathcal{M}^*(\beta I+\tilde K)V_\mathcal{M} = V_\mathcal{M}^*\beta IV_\mathcal{M} + V_\mathcal{M}^* \tilde KV_\mathcal{M} = \beta I_\mathcal{M} +\tilde K_\mathcal{M},
$$
which implies that $
V_\mathcal{M}^*(\beta I+\tilde K)V_\mathcal{M}$ is sum of a nonnegative scalar multiple of Identity and  a positive compact operator on a Hilbert space $\mathcal{M}$ which does satisfy the property $\mathcal{N}$ and hence proves our assertion.
\end{proof}

\begin{lemma}
Let $K\in\mathcal{B}(\mathcal{H})$ be a positive compact operator  and $F\in\mathcal{B}(\mathcal{H})$ be a self-adjoint finite-rank operator. Then $K+F$ can have at most finitely many negative eigenvalues.
\end{lemma}

\begin{proof}
Since $F$ is a self-adjoint finite-rank operator, there is an orthonormal basis $B$ of $\mathcal{H}$ consisting of eigenvectors of $F$
corresponding to which it is diagonalizable. This allows us to write  $F$ as the difference of two positive finite-rank operators, $F_+$ and $F_-$ so that $F=F_+-F_-$.
Consider the set of all eigenvectors in $B$ corresponding to which $F_-$ has nonzero (positive) eigenvalues. Needless to say that they are finite in numbers. Define $H_-$ to be the span of these eigenvectors. It is trivial to observe that $H_-$ is a closed finite-dimensional subspace of $\mathcal{H}$ and $\mathcal{H} = H_- \oplus H_-^\perp$. We assume that the dimension of $H_-$ is $k$, that is, $\dim{H_-}=k.$ We claim that the total number of negative eigenvalues of $K+F$ does not exceed $k$.
To prove this claim, we first observe that $K+F$ can now be rewritten as 
$ K + (F_+ -F_-) = (K + F_+) -F_- = \tilde{K} -F_-$ where $\tilde{K}=K + F_+$ is positive compact operator on $\mathcal{H}$. Also, $\tilde{K} -F_-$ is a self-adjoint compact operator and thus there exists an orthonormal basis $\ff{B}$ of $\mathcal{H}$ consisting entirely of eigenvectors of $\tilde{K} -F_-$ corresponding to which $\tilde{K} -F_-$ is diagonalizable. We next observe that 
$$
\text{ for any } x\in H_-^\perp,\hspace{0.2cm} \inner{(\tilde{K} -F_-)x}{x}\geq 0;
$$
because $F_-(x) =0$ for every $x\in H_-^\perp$ and $\inner{\tilde{K}x}{x}\geq 0$ for each $x\in \mathcal{H}$ and hence for each $x\in H_-^\perp.$
We are now ready to prove our claim. Consider the set of all orthonormal eigenvectors in $\ff{B}$ corresponding to which $\tilde{K} -F_-$ has negative eigenvalues. By the way of contradiction let us assume that the cardinality of this set is strictly bigger than $k$. We fix some $m>k$ and extract $m$ eigenvectors from this set. Let the set of these extracted eigenvectors be $\{v_1,v_2,v_3,...,v_m\}$ and the corresponding eigenvalues be $\{\lambda_1, \lambda_2, \lambda_3,...,\lambda_m\}.$
Since $m>k$, there exists $\alpha_1, \alpha_2,...,\alpha_m$ not all zero such that $P_{H_-}(\sum_{i=1}^m\alpha_i v_i )=0$. Then
\begin{align*}
\langle {(\tilde{K}-F_-)\left(\sum_{i=1}^m\alpha_i v_i\right )},{\sum_{j=1}^m\alpha_j v_j }\rangle &=\langle {\sum_{i=1}^m\alpha_i\lambda_i v_i},{\sum_{j=1}^m\alpha_j v_j }\rangle \\
&=\sum_{i=1}^m|\alpha_i|^2\lambda_i <0.
\end{align*}
But this contradicts the fact that $\sum_{i=1}^m\alpha_i v_i \in H_-^\perp$; for we established that $\text{ for any } x\in H_-^\perp, \inner{(\tilde{K} -F_-)x}{x}\geq 0$. This proves our claim.
\end{proof}
This observation leads us directly to the following proposition.
\begin{prop}
Let $K\in\mathcal{B}(\mathcal{H})$ be a positive compact operator  and $F\in\mathcal{B}(\mathcal{H})$ be a self-adjoint finite-rank operator. Then for every $\alpha \geq 0$, $\alpha I+K+F$ satisfies the property $\mathcal{N}$.
\end{prop}

\begin{proof}
The assertion is trivial if $\alpha =0$; for then $K+F$ is a compact operator which satisfies the property $\mathcal{N}$. We assume that $\alpha>0$. Notice that $K+F$ is a self-adjoint compact operator on $\mathcal{H}$ and thus there exists an orthonormal basis $B$ of $\mathcal{H}$ consisting entirely of eigenvectors of $K+F$ corresponding to which it is diagonalizable. From the previous lemma, $K+F$ can have at most finitely many negative eigenvalues. Let $\{\lambda_1, \lambda_2,...,\lambda_n\}$ be the set of all negative eigenvalues of $K+F$ with $\{v_1,v_2,...,v_n\}$ as the corresponding eigenvectors in basis $B$; and let $\{\mu_\beta:\beta\in \Lambda\}$ be the set of all remaining nonnegative eigenvalues of $K+F$ with $\{w_\beta:\beta\in \Lambda\}$ as the corresponding eigenvectors in $B$. We have $B:= \{v_1,v_2,...,v_n\}\cup \{w_\beta:\beta\in \Lambda\}$ and the matrix $M_B(K+F)$ of $K+F$ with respect to $B$ is given by
$$
K+F = \begin{bmatrix}
\lambda_1 & & &\vdots & & &\\
& \ddots & &\vdots & &\text{\huge 0}  &\\
 &  & \lambda_n &\vdots & & &\\
\hdots& \hdots & \hdots & \hdots&\hdots &\hdots&\hdots\\
&  &  &\vdots & \ddots& &\\
&  \text{\huge 0}&  &\vdots &  & \mu_\beta&\\
&  &  &\vdots &  & &\ddots \\
\end{bmatrix}.
$$
Observing the fact that 
$$
\|K+F\|=\max \{\{|\lambda_i|\}_{i=1}^n\cup \{\mu_\beta\}_{\beta\in \Lambda}\},
$$
we proceed to show that $\alpha I +K+ F$ satisfies property $\mathcal{N}$. To accomplish this we distinguish cases:\\

\noindent
\textit{Case I.} If $\mu_{\hat \beta} =\max \{\{|\lambda_i|\}_{i=1}^n\cup \{\mu_\beta\}_{\beta\in \Lambda}\}$ for some $\hat \beta\in \Lambda$.
Needless to say that $\|K+F\| = \mu_{\hat\beta} = \|(K+F)(w_{\hat \beta})\|$.
Clearly then 
\begin {align*}
\alpha +\mu_{\hat \beta}  \geq \alpha +|\lambda_i| &\geq |\alpha+\lambda_i| \text{ for each } i\in \{1,2,...,n\}, \text{ and } \\
 \alpha +\mu_{\hat \beta} &\geq \alpha +\mu_{\beta} \text{ for   each }\beta\in\Lambda.
\end{align*}
It is now easy to convince ourselves that if $w_{\hat\beta}$ be the eigenvector corresponding to the eigenvalue $\mu_{\hat \beta}$ then 
$
\|\alpha I +K+F\| =\|\alpha + \mu_{\hat\beta}\|=\|(\alpha I +K+F) (w_{\hat \beta})\|
$
which implies that $\alpha I +K+F$ achieves its norm at $w_{\hat \beta}$.\\

\noindent
\textit{Case II.} If $|\lambda_m| =\max \{\{|\lambda_i|\}_{i=1}^n\cup \{\mu_\beta\}_{\beta\in \Lambda}\}$ for some $m\in \{1,2,...,n\}$.
In this case it is important to observe that 
$$
\sup\{\mu_\beta:\beta\in \Lambda\} = \max\{\mu_\beta:\beta\in \Lambda\};
$$
indeed the matrix $M_B(K+F)$ can be written as

$$
 K+F=\begin{bmatrix}
\lambda_1 & & &\vdots & & &\\
& \ddots & &\vdots & &\text{\huge 0}  &\\
 &  & \lambda_n &\vdots & & &\\
\hdots& \hdots & \hdots & \hdots&\hdots &\hdots&\hdots\\
&  &  &\vdots & & &\\
&  \text{\huge 0}&  &\vdots &  &\text{\huge 0}&\\
&  &  &\vdots &  & & \\
\end{bmatrix} + \begin{bmatrix}
 & & &\vdots & & &\\
& \text{\huge 0} & &\vdots & &\text{\huge 0}  &\\
 &  &  &\vdots & & &\\
\hdots& \hdots & \hdots & \hdots&\hdots &\hdots&\hdots\\
&  &  &\vdots & \ddots& &\\
&  \text{\huge 0}&  &\vdots &  & \mu_\beta&\\
&  &  &\vdots &  & &\ddots \\
\end{bmatrix},
$$ 
where the first matrix is compact. Consequently the second matrix is forced to be compact which implies that $\sup\{\mu_\beta:\beta\in \Lambda\} = \max\{\mu_\beta:\beta\in \Lambda\}$. Let $ \max\{\mu_\beta:\beta\in \Lambda\} =\mu_{\tilde \beta} $ for some $\tilde \beta \in \Lambda.$
It is then a trivial observation that 
$
\sup \{\{|\alpha +\lambda_i|\}_{i=1}^n\cup \{\alpha +\mu_\beta\}_{\beta\in \Lambda}\} = \max\{\alpha + \mu_{\tilde \beta}, |\alpha + \lambda_1|,...,|\alpha+\lambda_n|\}
$
which ascertains that the operator $\alpha I + K +F$  satisfies the property $\mathcal{N}$.
We conclude the proof by a note that $\alpha I + K +F$ need not necessarily be  positive for the proof to work. 
\end{proof}
This result is the key to the theorem that follows. The following result could be deduced from \cite[Theorem 3.23]{CN} but there are some gaps in their proof of \cite[Lemma 3.7]{CN} which is essential to their proof of \cite[Theorem 3.23]{CN}; so we provide an independent proof.
\begin{thm}
Let $K\in\mathcal{B}(\mathcal{H})$ be a positive compact operator  and $F\in\mathcal{B}(\mathcal{H})$ be a self-adjoint finite-rank operator. Then for every $\alpha \geq 0$, $\alpha I+K+F$ satisfies the property $\mathcal{AN}$.
\end{thm}

\begin{proof}
Let $\mathcal{M}$ be an arbitrary nonempty closed linear subspace of the Hilbert space $\mathcal{H}$ and $V_{\mathcal{M}}:\mathcal{M}\longrightarrow \mathcal{H}$ be the inclusion map from $\mathcal{M}$ to $\mathcal{H}$ defined as $V_{\mathcal{M}}(x) = x$ for each $x\in \mathcal{M}$.

Let us then define $T:= \alpha I+K+F$ so that we have $T^*=\alpha I+K+F$ and $T^*T = (\alpha I+K+F)^2 = (\alpha^2 I)+(2\alpha K+K^2) + (2\alpha F + FK +KF +F^2) = \beta I+\tilde K + \tilde F$ where $\beta=\alpha^2 \geq 0$, $\tilde K = 2\alpha K+K^2$ and $\tilde F = 2\alpha F + FK +KF +F^2$ are respectively positive compact and self-adjoint finite-rank operators. Observe that
\begin{multline*}
 TV_\mathcal{M} \text{ is } \mathcal{N} \iff  (TV_\mathcal{M})^*(TV_\mathcal{M}) \text{ is } \mathcal{N}\\ \iff V_\mathcal{M}^*(T^*T)V_\mathcal{M} \text{ is } \mathcal{N} \iff  V_\mathcal{M}^*( \beta I+\tilde K + \tilde F)V_\mathcal{M} \text{ is } \mathcal{N}.
\end{multline*}
It suffices to show that $V_\mathcal{M}^*( \beta I+\tilde K + \tilde F)V_\mathcal{M}$  is  $\mathcal{N}$; for then, since $\mathcal{M}$ is arbitrary, it immediately follows from lemma \ref{Trick}  that $T$ is an $\mathcal{AN}$ operator. To this end, notice that $V_\mathcal{M}^*( \beta I+\tilde K + \tilde F)V_\mathcal{M}:\mathcal{M}\longrightarrow \mathcal{M}$ is an operator on $\mathcal{M}$ and 
$$
V_\mathcal{M}^*( \beta I+\tilde K + \tilde F)V_\mathcal{M} = V_\mathcal{M}^*\beta IV_\mathcal{M} +V_\mathcal{M}^*\tilde K V_\mathcal{M} +V_\mathcal{M}^* \tilde FV_\mathcal{M} = \beta I_\mathcal{M} +\tilde K_\mathcal{M}+\tilde F_\mathcal{M},
$$
which implies that $V_\mathcal{M}^*( \beta I+\tilde K + \tilde F)V_\mathcal{M}$ is sum of a nonnegative scalar multiple of Identity, a positive compact operator and a self-adjoint finite-rank operator on a Hilbert space $\mathcal{M}$ which, by the preceding proposition, satisfies the property $\mathcal{N}$. This proves the assertion.
\end{proof}

\begin{remark}
It is desirable at this stage to make an important remark: the sum of two $\mathcal{AN}$ operators need not necessarily be an $\mathcal{AN}$ operator. An example \cite[Section 2, Page 182]{CN} appears in \cite{CN} which establishes that the sum of two $\mathcal{N}$ operators need not necessarily be an $\mathcal{N}$ operator. In fact, one can show that the operators they consider are not just $\mathcal{N}$ operators but $\mathcal{AN}$. In what follows, we give an example of an operator $T\in \mathcal{H}$ which is $\mathcal{AN}$ but $2$Re$(T)$ is not, which also implies that sum of two $\mathcal{AN}$ operators need not be $\mathcal{AN}$.
\end{remark}

\begin{exam}\label{ANisNOTvs}
Let $\{e_i\}_{i\in \mathbb{N}}$ be the canonical orthonormal basis of the Hilbert space $\ell^2(\mathbb{N})$, $a\in (0,1]$, and $(a_i)_{i\in \mathbb{N}}, (b_i)_{i\in \mathbb{N}}$ be two sequences of real numbers such that
$$
0<a_1<a_2<...<a, \hspace{0.5cm} a_i\nearrow a, \hspace{0.2cm} \text{and} \hspace{0.2cm} a_i^2+b_i^2 = 1.
$$
Let $T\in \mathcal{B}(\ell^2(\mathbb{N}))$ defined as $Te_i =\lambda_ie_i$ for each $i\in \mathbb{N}$, where $\lambda_i=a_i+ib_i$. Then $T^*e_i= \overline\lambda_ie_i$. It is easy to observe that both $T$ and $T^*$ are isometries.  Indeed, if $x\in\ell^2(\mathbb{N})$, then $x=\sum_{i=1}^\infty\inner{x}{e_i}e_i$ which implies that
$$
\|Tx\|^2 =\left\|\sum_{i=1}^\infty\inner{x}{e_i}\lambda_ie_i\right\|^2=\|x\|^2
=\left\|\sum_{i=1}^\infty\inner{x}{e_i}\overline\lambda_ie_i\right\|^2= \|T^*x\|^2.
$$
By Proposition \ref{Isometry is AN}, we infer that $T$ and $T^*$ are $\mathcal{AN}$ operators.
We now show that $T+T^*$ is not an $\mathcal{AN}$ operator. Since every $\mathcal{AN}$ operator is an $\mathcal{N}$ operator, it suffices to show that  $T+T^*$ is not an $\mathcal{N}$ operator. To this end, notice that $\|T+T^*\| \geqslant \sup\{\|Te_i\|:i\in \mathbb{N}\}
  =\sup\{|\lambda_i+\overline\lambda_i|:i\in \mathbb{N}\}=
\sup\{|2a_i|:i\in\mathbb{N}\}=2a$. But for every $x\in\ell^2(\mathbb{N})$ with $\|x\|=1$, we have
$$
\|(T+T^*)x\|^2=\sum_{i=1}^\infty|\lambda_i+\overline\lambda_i|^2|\inner{x}{e_i}|^2
=\sum_{i=1}^\infty|2a_i|^2|\inner{x}{e_i}|^2<4a^2.
$$ 
Consequently, for every $x\in\ell^2(\mathbb{N})$ of unit length 
$\|(T+T^*)x\|<2a\leqslant \|T+T^*\|$ which implies that $T+T^*$ does not satisfy the property $\mathcal{N}$.

\end{exam}

\section{Spectral Characterization Of Positive $\mathcal{AN}$ Operators}

The final theorem of the preceding section just established ------ that for every $\alpha\geq 0, \alpha I + K +F$ satisfies the $\mathcal{AN}$ property where $K$ and $F$ are respectively positive compact and self-adjoint finite-rank operators ------ is the stronger version of the backward implication of our spectral theorem for positive $\mathcal{AN}$ operators. If the operator $ \alpha I + K +F$ is also positive then the implication can be reversed and the two conditions are equivalent. This is what the next theorem states.

\begin{thm}[Spectral Theorem For positive $\mathcal{AN}$ Operators]\label{SpThPAN}
Let $\mathcal{H}$ be a complex Hilbert space of arbitrary dimension and let $P$ be a positive operator on $\mathcal{H}$. Then $P$ is an $\mathcal{AN}$ operator if and only if $P$ is of the form $P= \alpha I +K+F$, where $\alpha \geq 0, K$ is a positive compact operator and $F$ is self-adjoint finite-rank operator.
\end{thm}

\begin{proof}
It suffices to prove the forward implication. We assume that $P\in\mathcal{B}(\mathcal{H}) $ is a positive $\mathcal{AN}$ operator. Theorem \ref{NC} asserts that there exists an orthonormal basis $B=\{v_\lambda:\lambda\in \Lambda\}$ consisting entirely of eigenvectors of $P$ and for every $\lambda \in \Lambda, Tv_\lambda =\beta_\lambda v_\lambda$ with $\beta_\lambda \geq 0$. A moment's thought will convince the reader that there are four mutually exclusive and exhaustive set of possibilities for the spectrum $\sigma (P)= \text{clos}\{\beta_\lambda:\lambda\in \Lambda\}$ of $P$.\\

\noindent
\textit{Case 1.} $\sigma (P)$ has neither a limit point nor an eigenvalue with infinite multiplicity. \\ The index set $\Lambda$ is then finite; for if it is not then the set $\{\beta_\lambda:\lambda\in \Lambda\}$ (counting multiplicities) of eigenvalues is also infinite. Since each eigenvalue in this set can have at most finite multiplicity, it is obvious then that the set $\{\beta_\lambda:\lambda\in \Lambda\}$ (without counting multiplicities) of distinct eigenvalues of $P$ is infinite. More interestingly,  $\{\beta_\lambda:\lambda\in \Lambda\}$ is bounded above by the operator norm of $P$ and below by $0$. Since every infinite bounded subset of real numbers has a limit point, we arrive at a contradiction and hence $\Lambda$ is finite. This forces the Hilbert space $\mathcal{H}$ to be finite dimensional. In that case $P$ boils down to a positive (and hence self-adjoint) finite-rank operator and we can safely assume that $P= \alpha I +K+F$ with $\alpha =0$, $K=0$ and $F$ the operator in question. \\ 

\noindent
\textit{Case 2.} $\sigma (P)$ has no limit point but has one eigenvalue with infinite multiplicity.\\
Let $\beta_0 \in \{\beta_\lambda:\lambda\in \Lambda\}$ be the eigenvalue with infinite multiplicity. Then the set $\Gamma:=\Lambda \setminus \{\lambda\in \Lambda:\beta_\lambda=\beta_0\}$ is finite; for if it is not, then the set $\{\beta_\lambda:\lambda\in \Gamma\}$ (counting multiplicities) is also infinite which in turn implies that the set $\{\beta_\lambda:\lambda\in \Gamma\}$ (without counting multiplicities) is infinite because each eigenvalue in this set can have at most finite multiplicity. Since $\{\beta_\lambda:\lambda\in \Gamma\}$ is bounded and every infinite bounded subset of real numbers has a limit point, we arrive at a contradiction.This implies that $\Gamma$ is finite. 
Observe that for an arbitrary $x\in \mathcal{H}$, we have 
$
x=\sum_{\lambda\in \Lambda}\inner{x}{v_\lambda}v_\lambda=\sum_{\lambda\in \Gamma}\inner{x}{v_\lambda}v_\lambda+\sum_{\lambda\in\Lambda\setminus\Gamma}\inner{x}{v_\lambda}v_\lambda
$
so that for every $x\in \mathcal{H},$
\begin{align*}
Px &=\sum_{\lambda\in \Gamma}\inner{x}{v_\lambda}P(v_\lambda)
+\sum_{\lambda\in\Lambda\setminus\Gamma}\inner{x}{v_\lambda}P(v_\lambda)\\
&=\sum_{\lambda\in \Gamma}\inner{x}{v_\lambda}\beta_\lambda v_\lambda
+\sum_{\lambda\in\Lambda\setminus\Gamma}\inner{x}{v_\lambda}\beta_0 v_\lambda\\
&=\sum_{\lambda\in \Gamma}(\beta_\lambda -\beta_0)\inner{x}{v_\lambda}v_\lambda
+\beta_0 \sum_{\lambda\in\Lambda}\inner{x}{v_\lambda}v_\lambda\\ 
&=\sum_{\lambda\in \Gamma}(\beta_\lambda -\beta_0)(v_\lambda \otimes v_\lambda) (x)
+\beta_0 Ix\\
&= \left( \sum_{\lambda\in \Gamma}(\beta_\lambda -\beta_0)(v_\lambda \otimes v_\lambda) + \beta_0I\right)(x).
\end{align*}
To conclude this case it suffices to observe that $\beta_0 \geq 0$ and $\sum_{\lambda\in \Gamma}(\beta_\lambda -\beta_0)(v_\lambda\otimes v_\lambda)$ is a self-adjoint finite-rank operator. It then readily follows that $P= \alpha I +K+F$, where $\alpha = \beta_0, K=0$ and $F =\sum_{\lambda\in \Gamma}(\beta_\lambda -\beta_0)(v_\lambda \otimes v_\lambda).$ \\

\noindent
\textit{Case 3.} $\sigma (P)$ has no eigenvalue with infinite multiplicity but has a limit point.\\
The index set $\Lambda$ is then countable; for if it is uncountable then the set $\{\beta_\lambda:\lambda\in \Lambda\}$ (counting multiplicities) is also uncountable thereby rendering the set $\{\beta_\lambda:\lambda\in \Lambda\}$ (without counting multiplicities) uncountable since each eigenvalue in this set has finite multiplicity. Then this uncountable set must have at least two limit points; and since this is impossible, we infer that $\Lambda$ is countable and hence $\mathcal{H}$ is separable. Having shown that $\Lambda$ is countable, we can safely replace $\Lambda$ by $\mathbb{N}$. This essentially redefines the spectrum 
$\sigma(P)=\text{clos}\{\beta_n:n\in \mathbb{N}\}$ of $P$. 

Now let $\beta \in \sigma(P)$ be the unique limit point in the spectrum. We wish to  reorder the elements of $\{\beta_n:n\in \mathbb{N}\}$ linearly in accordance with their size. To accomplish this, we first notice that there are at most only finitely many terms of the set $\{\beta_n:n\in \mathbb{N}\}$ that are strictly less than $\beta$------ represent this set of finite elements by $\{\beta_1,\beta_2,...,\beta_k\}$ counting multiplicities. We next consider the set $\{\beta_n:\beta_n>\beta\}_{n\in \mathbb{N}}$         of all terms that are strictly bigger than $\beta$. We then inductively define a nonincreasing sequence $(\beta_{k+m})_{m\in \mathbb{N}}$ as 
\begin{align*}
&\beta_{k+1}:=\max\{\beta_n:\beta_n>\beta\}_{n\in \mathbb{N}},\\
&\beta_{k+2}:=\max\{\beta_n:\beta_n>\beta\}_{n\in \mathbb{N}}\setminus \{\beta_{k+1}\},\\
& \vdots\\
&\beta_{k+m}:=\max\{\beta_n:\beta_n>\beta\}_{n\in \mathbb{N}},\setminus \{\beta_{k+1},...,\beta_{k+m-1}\},\\
&\vdots
\end{align*}
This decreasing sequence is bounded below by $\beta$, so it converges to $\beta$; for if it converges to any other point-----which, in that case, happens to be a limit point of $\sigma(P)$----- then that contradicts the existence of only one limit point in the spectrum.\\
 Before we go further, it is worth establishing that the set $\{\beta_n:\beta_n>\beta\}_{n\in \mathbb{N}}$ of all eigenvalues of $P$ has been exhausted in the process of constructing the sequence $(\beta_{k+m})_{m\in \mathbb{N}}$, that is, each eigenvalue of $P$ that is strictly bigger than $\beta$ is a term of the sequence $(\beta_{k+m})_{m\in \mathbb{N}}$. This is rather a trivial observation if we show that whenever $\beta_n>\beta$ is an eigenvalue of $P$ there exist only finitely many $j$'s such that $\beta_j>\beta_n$. Now suppose, on the contrary, that there are infinitely many such $j$'s. Then they form an infinite bounded set of real numbers with a limit point greater than or equal to $ \beta_n$. But since $\beta_n>\beta$, it contradicts the fact that $\beta$ is the unique limit point of the $\sigma(T)$.\\
This inductive method of constructing the decreasing sequence is exhaustive too and as an immediate consequence we re-order the eigenvalues of $P$:
$$
 \{\beta_n\}_{n=1}^k \cup \{\beta_n\}_{n=k+1}^\infty; \text{ where } \{\beta_n\}_{n=k+1}^\infty \text{ converges to } \beta.
$$ 
Let us rename and denote by $\{v_n\}_{n=1}^k$ and $\{w_n\}_{n=k+1}^\infty$ the eigenvectors corresponding to the eigenvalues $\{\beta_n\}_{n=1}^k $ and $\{\beta_n\}_{n=k+1}^\infty$ respectively.
Observe that for an arbitrary $x\in \mathcal{H}$, we have 
$
x=\sum_{n=1}^k\inner{x}{v_n}v_n
+\sum_{n=k+1}^\infty \inner{x}{w_n}w_n
$
so that for every $x\in \mathcal{H},$
\begin{align*}
Px &=\sum_{n=1}^k\inner{x}{v_n}P(v_n)+\sum_{n=k+1}^\infty \inner{x}{w_n}P(w_n)\\
&=\sum_{n=1}^k\inner{x}{v_n}\beta_n v_n+\sum_{n=k+1}^\infty \inner{x}{w_n}\beta_n w_n\\
&=\sum_{n=1}^k(\beta_n-\beta)\inner{x}{v_n} v_n+\sum_{n=k+1}^\infty (\beta_n-\beta)\inner{x}{w_n} w_n + \beta I x \\
&=\left(\sum_{n=1}^k(\beta_n-\beta)(v_n\otimes v_n)+\sum_{n=k+1}^\infty (\beta_n-\beta)(w_n\otimes w_n) + \beta I \right) (x). \\
\end{align*}
To conclude this case it suffices to observe that $\beta \geq 0$, $\sum_{n=1}^k(\beta_n-\beta)(v_n\otimes v_n)$ is a self-adjoint finite-rank operator, and $\sum_{n=k+1}^\infty (\beta_n-\beta)(w_n\otimes w_n)$ is a positive compact operator. It then readily follows that $P= \alpha I +K+F$, where $\alpha = \beta, K=\sum_{n=k+1}^\infty (\beta_n-\beta)(w_n\otimes w_n)$ and $F =\sum_{n=1}^k(\beta_n-\beta)(v_n\otimes v_n).$ \\

\noindent
\textit{Case 4.} $\sigma (P)$ has both a limit point and an eigenvalue with infinite multiplicity.\\
Let $\beta\in \{\beta_\lambda:\lambda \in \Lambda\}$ be the unique eigenvalue with infinite multiplicity which compels it to be the unique limit point of the spectrum $\sigma(P)$ of $P$. That the set $\Gamma:=\Lambda \setminus \{\lambda:\beta_\lambda=\beta\}$ is countable is, at this stage, a trivial observation. This leaves us with$\{\beta_\lambda:\lambda \in \Lambda\} =\{\beta_\lambda:\lambda\in \Gamma\}\cup\{\beta\}.$ Since $\{\beta_\lambda:\lambda\in \Gamma\}$ is countable, by the argument in the previous case, we can reorder the eigenvalues of this set in such a way that for some $k\in \mathbb{N}$, $$
\{\beta_\lambda:\lambda\in \Gamma\} = \{\beta_n\}_{n=1}^k \cup \{\beta_n\}_{n=k+1}^\infty \cup \{\beta\},
$$
where, by the constructive method discussed previously, $ \{\beta_n\}_{n=1}^k$ (counting multiplicities) is the set of all eigenvalues strictly less than $\beta$ and $\{\beta_n\}_{n=k+1}^\infty$ is a nonincreasing sequence converging to $\beta$. We next rename and denote by $\{v_n\}_{n=1}^k, \{w_n\}_{n=k+1}^\infty,$ and $\{v_\lambda\}_{\lambda\in \Lambda\setminus \Gamma}$ the eigenvectors corresponding to the eigenvalues $\{\beta_n\}_{n=1}^k,\{\beta_n\}_{n=k+1}^\infty$, and $\{\beta_\lambda\}_{\lambda\in \Lambda\setminus \Gamma}$ respectively. Observe that for an arbitrary $x\in \mathcal{H}$, we have
$x=\sum_{n=1}^k\inner{x}{v_n}v_n
+\sum_{n=k+1}^\infty \inner{x}{w_n}w_n +\sum_{\lambda\in \Lambda\setminus \Gamma}\inner{x}{v_\lambda}v_\lambda.$
This yields that for every $x\in \mathcal{H}$
\begin{align*}
Px &=\sum_{n=1}^k\inner{x}{v_n}P(v_n)+\sum_{n=k+1}^\infty \inner{x}{w_n}P(w_n) +\sum_{\lambda\in \Lambda\setminus \Gamma}\inner{x}{v_\lambda}P(v_\lambda)\\
&=\sum_{n=1}^k\inner{x}{v_n}\beta_n v_n+\sum_{n=k+1}^\infty \inner{x}{w_n}\beta_n w_n + \sum_{\lambda\in \Lambda\setminus \Gamma}\inner{x}{v_\lambda}\beta v_\lambda\\
&=\sum_{n=1}^k(\beta_n-\beta)\inner{x}{v_n} v_n+\sum_{n=k+1}^\infty (\beta_n-\beta)\inner{x}{w_n} w_n \\
&\hspace{2.5cm}+ \sum_{\lambda\in \Gamma}\inner{x}{v_\lambda}\beta v_\lambda + \sum_{\lambda\in \Lambda\setminus \Gamma}\inner{x}{v_\lambda}\beta v_\lambda\\
&=\left(\sum_{n=1}^k(\beta_n-\beta)(v_n\otimes v_n)+\sum_{n=k+1}^\infty (\beta_n-\beta)(w_n\otimes w_n) + \beta I \right) (x) .\\
\end{align*}
It then immediately follows that $P= \alpha I +K+F$, where $\alpha = \beta, K=\sum_{n=k+1}^\infty (\beta_n-\beta)(w_n\otimes w_n)$ and $F =\sum_{n=1}^k(\beta_n-\beta)(v_n\otimes v_n).$ \\
We complete the proof by observing that in all the four possibilities, we get the desired form. 
\end{proof}

Example \ref{ANisNOTvs} establishes the fact that the class of $\mathcal{AN}$ operators is not closed under addition. However, it is easy to see that it is closed under scalar multiplication, that is, if $T\in \mathcal{B}(\mathcal{H},\mathcal{K})$ is $\mathcal{AN}$ and $\alpha\in \mathbb{C}$, then $\alpha T$ is also $\mathcal{AN}$; for if $\mathcal{M}$ is an arbitrary nontrivial closed subspace of $\mathcal{H}$, then $\|\alpha TV_\mathcal{M}\|=|\alpha|\|TV_\mathcal{M}\|=|\alpha|\|TV_\mathcal{M}(x_0)\|=\|\alpha TV_\mathcal{M}(x_0)\|,$ where $x_0\in \mathcal{M}, \|x_0\|=1,$ and $\|TV_\mathcal{M}(x_0)\|=\|TV_\mathcal{M}\|.$

If we consider the class $\mathcal{B}(\mathcal{H})_{\mathcal{AN}^+}$ of positive $\mathcal{AN}$ operators, what can be said about it in the similar vein? To answer this question, let $T_1,T_2 \in \mathcal{B}(\mathcal{H})_{\mathcal{AN}^+}$. It is fairly obvious that $T_1+T_2$ is positive. Moreover, by Theorem \ref{SpThPAN},
$T_1= \alpha_1I +K_1+F_1$ and $T_2=\alpha_2I+K_2+F_2$ where $\alpha_1, \alpha_2\geq 0; K_1,K_2$ are positive compact operators, and $F_1,F_2$ are self-adjoint finite-rank operators. Then $T_1+T_2= (\alpha_1+\alpha_2)I +(K_1+K_2)+(F_1+F_2)$ and hence it is $\mathcal{AN}.$
Also, if $c\in \mathbb{R}, c\geq 0,$ then $cT_1\in \mathcal{B}(\mathcal{H})_{\mathcal{AN}^+}.$ Finally, if $T$ and $-T$ are both in $\mathcal{B}(\mathcal{H})_{\mathcal{AN}^+}$, then $\inner{Tx}{x}\geq 0$
and $\inner{Tx}{x}\leq 0$ which implies that $\inner{Tx}{x}=0$ for each $x\in\mathcal{H}$ and so $T=0$. These observations, together with the fact that $\mathcal{B}(\mathcal{H})_{sa}:=\{T\in \mathcal{B}(\mathcal{H}): T=T^*\}$ is a real Banach space, implies that 
$\mathcal{B}(\mathcal{H})_{\mathcal{AN}^+}$ is a cone in $\mathcal{B}(\mathcal{H})_{sa}$, which is proper in the sense that $\mathcal{B}(\mathcal{H})_{\mathcal{AN}^+}\cap (-\mathcal{B}(\mathcal{H})_{\mathcal{AN}^+})=\{0\}$.\\

\section{Spectral Characterization Of $\mathcal{AN}$ Operators}

For any operator $T\in \mathcal{B}(\mathcal{H},\mathcal{K})$, we know that $T^*T\in \mathcal{B}(\mathcal{H})$ and $T^*T\geq0.$ Moreover, there exists a unique positive operator 
$|T|:=\sqrt{T^*T}$ such that $|T|^2=T^*T$. We state the polar decomposition theorem, which is a standard theorem and its proof is thus omitted.

\begin{prop}[Polar Decomposition Theorem]\label{PDT}
Let $\mathcal{H}, \mathcal{K}$ be  complex Hilbert spaces. If $T\in \mathcal{B}(\mathcal{H},\mathcal{K})$, then there exists a  unique partial 
isometry $U:\mathcal{H}\longrightarrow \mathcal{K}$ with final space $\text{clos}(\text{ran}{T})$ and initial space $\text{clos}(\text{ran}{|T|})$ such that $T=U|T|$ and $|T|=U^*T$. If $T$ is invertible, then $U$ is unitary.
\end{prop}
 The following lemma is the key to the main theorem of this section.
\begin{lemma} 
Let $\mathcal{H}$ and $\mathcal{K}$ be complex Hilbert spaces and let $T\in \mathcal{B}(\mathcal{H},\mathcal{K})$. Then $T$ is $\mathcal{AN}$ iff $|T|$ is $\mathcal{AN}$.
\end{lemma}

\begin{proof}
Let $\mathcal{M}$ be an arbitrary nontrivial closed subspace of $\mathcal{H}$. For any $x\in \mathcal{M}$ notice that 
\begin{multline*}
\|T|_{\mathcal{M}}(x)\| =\|Tx\| = \sqrt{\inner{Tx}{Tx}}= \sqrt{\inner{T^*Tx}{x}}\\
= \sqrt{\inner{|T|^2x}{x}} = \sqrt{\inner{|T|x}{|T|x}} ==\||T|(x)\| = \||T||_{\mathcal{M}}(x)\|,
\end{multline*}
which essentially guarantees that 
$$
\|T|_{\mathcal{M}}\| =\||T||_{\mathcal{M}}\|.
$$
Since $\mathcal{M}$ is arbitrary, the assertion follows.
\end{proof}

\begin{exam}\label{coisom}  Let $V: l^2 \to l^2$ be an isometry onto a subspace $\mathcal M$ with infinite codimension. By Proposition~\ref{Isometry is AN}, $V$ is $\mathcal{AN}.$ But $|V^*|= VV^*=P_{\mathcal M}$ is the orthogonal projection onto $\mathcal M$ and since $P_{\mathcal M}$ has two eigenvalues of infinite multiplicity, $P_{\mathcal M}$ is not $\mathcal{AN}$ by Corollary~\ref{oneinfmult}.  Thus, $V$ is $\mathcal{AN}$ but $V^*$ is not $\mathcal{AN}.$
\end{exam}

By the preceding lemma, the polar decomposition theorem and the spectral theorem for positive $\mathcal{AN}$ operators, we can safely consider the following theorem to be fully proved.
\begin{thm}[Spectral Theorem For $\mathcal{AN}$ Operators]\label{SpThAN}
Let $\mathcal{H}$ and $\mathcal{K}$ be complex Hilbert spaces of arbitrary dimensions and let $T\in \mathcal{B}(\mathcal{H},\mathcal{K})$ such that $|T|=U^*T$, where $U$ is the unique partial 
isometry $U:\mathcal{H}\longrightarrow \mathcal{K}$ with final space $\text{clos}(\text{ran}{T})$ and initial space $\text{clos}(\text{ran}{|T|})$.
Then $T$ is an $\mathcal{AN}$ operator if and only if $U^*T$ is of the form $U^*T= \alpha I +F+K$, where $\alpha \geq 0, K$ is a positive compact operator and $F$ is self-adjoint finite-rank operator. 
\end{thm}


\begin{thebibliography}{1}

\bibitem{CN}
X.~Carvajal and W.~Neves, \emph{Operators that achieve the norm}, Integral
  Equations and Operator Theory \textbf{72} (2012), no.~2, 179--195.

\bibitem{Halmos}
P.~Halmos, \emph{A hilbert space problem book}, Springer-Verlag, New York,
  1982.

\bibitem{KR}
R.~V. Kadison and J.~R. Ringrose, \emph{Fundamentals of the theory of operator
  algebras. {S}pecial {T}opics {V}olume {III}: Elementary theory - {A}n
  exercise approach}, Birkhäuser Boston, Boston, MA, 1991.

\bibitem{R}
G.~Ramesh, \emph{Structure theorem for $\mathcal{AN}$-operators}, Journal of the Australian
  Mathematical Society \textbf{96} (2014), no.~3, 386--395.

\end{thebibliography}

\providecommand{\bysame}{\leavevmode\hbox to3em{\hrulefill}\thinspace}
\providecommand{\MR}{\relax\ifhmode\unskip\space\fi MR }
\providecommand{\MRhref}[2]{%
  \href{http://www.ams.org/mathscinet-getitem?mr=#1}{#2}
}
\providecommand{\href}[2]{#2}

\end{document}